\RequirePackage{fix-cm}
\documentclass{svjour3}                     

\smartqed  
\usepackage{graphicx}
\usepackage{amsmath}
\usepackage{amssymb}
\usepackage{latexsym}
\usepackage{exscale}
\usepackage{eucal}
\DeclareMathOperator{\Vol}{Vol}
\DeclareMathOperator{\sign}{sign}
\numberwithin{equation}{section}
\numberwithin{theorem}{section}
\numberwithin{lemma}{section}
\numberwithin{remark}{section}
\numberwithin{corollary}{section}
\numberwithin{example}{section}

\def\subclassname{{\bfseries Mathematics Subject Classification
(2010)}\enspace}
\def\subclass#1{\par\addvspace\medskipamount{\rightskip=0pt plus1cm
\def\and{\ifhmode\unskip\nobreak\fi\ $\cdot$
}\noindent\subclassname\ignorespaces#1\par}}

\journalname{Acta Mathematica Vietnamica}
\begin{document}

\title{Some inequalities for partial derivatives on time scales}


\author{Tran Dinh Phung
}


\institute{ Tran Dinh Phung \at
            PhD student at Department of Mathematics,
Quy Nhon University, Binh Dinh, Vietnam \\
              \email{trandinhphung@qnu.edu.vn; trandinhphung89@gmail.com}  \\
               This work is funded by Vietnam National Foundation for Science and Technology Development (NAFOSTED) under grant number 101.02-2014.32. 
}         

\date{Received: date / Accepted: date}

\maketitle

\begin{abstract}
We first prove some weighted inequalities for compositions of functions on time scales which are in turn applied to establish some new dynamic Opial-type inequalities in several  variables. Some generalizations and applications to partial differential dynamic equations are also considered. 
 \subclass{26D15\and 26D10\and 26E70}
\end{abstract}

\section{Introduction}
In  1960, Opial \cite{O} established the following integral inequality
\begin{equation}
\int_0^b|f(x)f'(x)|dx \le \frac{b}{4}\int_0^b |f'(x)|^2dx,
\label{1.1}
\end{equation}
 where $f$ is absolutely continuous on $[0,b]$ such that $f(0)=f(b)=0.$  In 1962, Beesack \cite{B} showed the following result which implies \eqref{1.1}  and is very useful in applications: If $f$ is  absolutely continuous on $[0,b]$ and $f(0)=0$, then 
\begin{equation}\label{eqO2}
\int_0^b|f(x)f'(x)|dx \le \frac{b}{2}\int_0^b |f'(x)|^2dx.
\end{equation}

Since then, many generalizations of  Opial's inequality \eqref{eqO2} in various directions have been given, one of which was given in 1967 by Godunova and Levin  \cite{Go}: Let $F$ be a convex and increasing function on $[0, \infty)$ with $F(0) = 0$, and  $f$ be a real-valued absolutely continuous function defined on $[a,b]$ with $f(a) = 0$; then
\begin{equation}\label{eq1.2}
\int_a^bF'(|f(x)|)|f'(x)|dx \le F\bigg( \int_a^b|f'(x)|dx\bigg).
\end{equation}
 Later, some  multidimentional generalizations of \eqref{eq1.2} were given, such as  Pe\v{c}ari\'{c}  \cite{pe},  Pachpatte  \cite{pach}, and  Andri\'{c} et al. \cite{AND}. 

In 2015,  Duc, Nhan and Xuan \cite{NDX} extended and generalized  \eqref{eqO2} to  several independent variables and  demonstrated the usefulness  in the field of partial differential equations:
\begin{equation}\label{eq:}
\begin{split}
&\bigg[ \int_{\Omega}\bigg| \partial^{\boldsymbol{\alpha}}  \bigg( \prod^m_{j=1}{  G_j(u_j(\boldsymbol{x})) } \bigg) \bigg|^s K_{\boldsymbol{\alpha}}(\boldsymbol{b}, \boldsymbol{x})\sigma(\boldsymbol{x})     d \boldsymbol{x} \bigg]^{1/s}\\
 &\leq C \prod^m_{j =1}{ \bigg[ G_j \bigg(  \int_{\Omega }  | \partial^{\boldsymbol{\alpha}}u_j(\boldsymbol{x}) |^p  K_{\boldsymbol{\alpha}}(\boldsymbol{b}, \boldsymbol{x})\rho_j(\boldsymbol{x})   d \boldsymbol{x}  \bigg) \bigg]^{1/p}  },  \\
\end{split}
\end{equation}
where $C$ is a constant.

In  recent years, the theory of time scales which was introduced by  Hilger \cite{Hilger} in order to unify the study of differential and difference equations, has  received a lot of attention. The readers may find much of time scales calculus in books by Bohner and Peterson \cite{BP}, \cite{BP1}. One of main subjects of the qualitative analysis on time scales is to prove some new dynamic inequalities. Opial-type inequalities on time scales was first proved by Bohner and Kaymak\c{c}alan \cite{BK} in 2001  (see also \cite{ABP}), in which they showed that if $f: [0, b]_{\mathbb{T}}\to  \mathbb{R}$  is delta differentiable with $ f(0) = 0$, then
\begin{equation}
\int_0^b|[f(x)+f^{\sigma}(x)]f^{\Delta}(x) | \Delta x \le b\int_0^b |f^{\Delta}(x)|^2\Delta x. 
\label{eq1.1}
\end{equation}
Afterwards, numerous authors have studied variants of  \eqref{eq1.1} (see, for example, \cite{Ka}, \cite{Lihan}, \cite{SAK}, \cite{SAK1}, \cite{SORA}, \cite{SIVA}, and \cite{YZ}). The best  reference here is the book by Agarwal, O'Regan, and Saker  \cite[Chapter 3]{AOS}, where  the most popular articles on this subject are collected.  

However, to the best of the author knowledge nothing is known regarding Opial-type inequalities involving functions of several variables and their partial derivatives on time scales. Thus, the aim of this paper is to study some weighted integral inequalities for delta derivatives acting on compositions of functions on time scales which are in turn applied to establish multidimentional dynamic Opial-type inequalities. As applications, we establish Lyapunov-type inequalities for half-linear dynamic equations and obtain upper bounds of solutions of certain integro-partial dynamic equations.

\section{Preliminaries}
In the most part we assume the readers were so familiar with basic time scales calculus. More information about time scales calculus can be found in \cite{BP} and \cite{BP1}. In this section, we only present some basic definitions and notations about calculus in several variables on time scales.

Let $\mathbb{T}$ be a time scale, and let $\sigma, \rho,$ and $\Delta$ denote, respectively,  the forward jump, backward jump,  and delta operator on $\mathbb{T}$. 
Fix $n \in \mathbb{N}$ and let  $\mathbb{T}_j,$ where $ j =1, ..., n,$ be time scales, and 
 \[\Lambda^n = \mathbb{T}_1\times \cdot \cdot \cdot\times\mathbb{T}_n = \{ \boldsymbol{x}= (x_1, ..., x_n): x_j \in \mathbb{T}_j \,\,\, \text{for all}\,\,\,  j\in [1, n]_{\mathbb{N}}\}\]
  be the $n$-dimensional time scale. 
For $i \in [1, n]_{\mathbb{N}},$ let  $\sigma_j, \rho_j,$ and $\Delta_j$ denote, respectively,  the forward jump, backward jump,  and delta operator on $\mathbb{T}_j$.  We define ${\mathbb T}_j^{\kappa} = {\mathbb T}_j$ if ${\mathbb T}_j$ does not have a left-scattered maximum ${x_j}_{\max}$; otherwise ${\mathbb T}_j^{\kappa} = {\mathbb T}_j\setminus\{{x_j}_{\max}\}$. 
The graininess functions $\mu_j:{\mathbb T}_j\to [0,\infty)$ are defined by $\mu_j(x_j)=\sigma_j(x_j)-x_j$ for $j \in [1, n]_{\mathbb{N}}$.
For $x_j \in \mathbb{T}_j, j \in [1, n]_{\mathbb{N}},$ we denote $\rho^2_j(x_j) = \rho_j(\rho_j(x_j))$ and $\rho^k_j(x_j) = \rho_j(\rho^{k-1}_j(x_j))$ for $k \in \mathbb{N}.$ For convenience we put $\rho_j^0(x_j) = x_j$, $ j\in [1, n]_{\mathbb{N}}.$
For $\boldsymbol{x} = (x_1, ..., x_n), \boldsymbol{y}= (y_1, ...,  y_n) \in \Lambda^n$, we shall write $\boldsymbol{x}\leq \boldsymbol{y}$ instead of $x_j \leq y_j$ for all $j\in [1,n]_{\mathbb{N}} $. Analogously one has to understand $\boldsymbol{x} = \boldsymbol{y}, \boldsymbol{x} > \boldsymbol{y}$ and $\boldsymbol{x} < \boldsymbol{y}$, respectively. We put $\boldsymbol{x} +\boldsymbol{y} = (x_1+y_1, ..., x_n+y_n)$.
We denote by  $\boldsymbol{\lambda}= (\lambda_1, ..., \lambda_n)$ the multi-index, i.e. $\lambda_j \in \mathbb{N}_0 = \mathbb{N}\cup \{0\},  j \in [1,n]_{\mathbb{N}}.$  In particular, let $\boldsymbol{1}=(1, ..., 1)$. 
 Let $\boldsymbol{a}= (a_1, ..., a_n)$,  $ \boldsymbol{b} =(b_1, ..., b_n)$, and ${\boldsymbol\rho}^{\boldsymbol{\lambda}-\boldsymbol{1}}(\boldsymbol{b}) = (\rho^{\lambda_1-1}_1(b_1), ..., \rho^{\lambda_n-1}_n(b_n))$ for $\boldsymbol{\lambda}\geq \boldsymbol{1},$ be in $\Lambda^n$ such that $\boldsymbol{a}< {\boldsymbol \rho}^{\boldsymbol{\lambda}-\boldsymbol{1}}(\boldsymbol{b})$. Then we set
\[ \Omega  = \{\boldsymbol{x} \in \Lambda^n: \boldsymbol{a}\leq \boldsymbol{x}\leq \boldsymbol{b} \}, \]
\[ \Omega^{\kappa^{\boldsymbol{\lambda-1}}}  = \{\boldsymbol{x} \in \Lambda^n: \boldsymbol{a}\leq \boldsymbol{x}\leq  {\boldsymbol\rho}^{\boldsymbol{\lambda}-\boldsymbol{1}}(\boldsymbol{b})  \}, \]
\[\Omega_{\boldsymbol{x}} =\{\boldsymbol{t} \in \Lambda^n: \boldsymbol{a}\leq \boldsymbol{t}\leq \boldsymbol{x} \}, \quad \boldsymbol{x} \in \Omega, \]
\[\bar{\Omega}_{\boldsymbol{x}} =\{\boldsymbol{t} \in \Lambda^n: \boldsymbol{x}\leq \boldsymbol{t}\leq {\boldsymbol \rho}^{\boldsymbol{\lambda}-\boldsymbol{1}}(\boldsymbol{b}) \}, \quad \boldsymbol{x} \in \Omega^{\kappa^{\boldsymbol{\lambda-1}}}, \]
\[ \Omega' =[a_2, b_2]_{\mathbb{T}_2} \times \cdot \cdot \cdot \times [a_n, b_n]_{\mathbb{T}_n}. \]
For any  real-valued rd-continuous  function $f$ defined on $\Omega$ we denote  by $\int_{\Omega}f(\boldsymbol{x})  \Delta \boldsymbol{x}$, $\int_{\Omega^{\kappa^{\boldsymbol{\lambda-1}}}}f(\boldsymbol{x})  \Delta \boldsymbol{x}$, 
$\int_{\Omega_{\boldsymbol{x}}}f(\boldsymbol{t})  \Delta \boldsymbol{t}$ for any $\boldsymbol{x} \in \Omega,$
$\int_{\bar{\Omega}_{\boldsymbol{x}}}f(\boldsymbol{t})  \Delta \boldsymbol{t}$ for any $\boldsymbol{x} \in \Omega^{\kappa^{\boldsymbol{\lambda-1}}},$ and
$\int_{\Omega'}f(\boldsymbol{x'})  \Delta \boldsymbol{x'}$  for  $\boldsymbol{x'} \in \Omega'$,  are
 $n-$fold integrals $\int^{b_1}_{a_1}\cdot \cdot \cdot \int^{b_n}_{a_n} f(x_1, ..., x_n)\Delta x_1 \cdot \cdot \cdot \Delta x_n,$ 
   $\int^{\rho^{\lambda_1-1}_1(b_1)  }_{a_1}\cdot \cdot \cdot \int^{\rho^{\lambda_n-1}_n(b_n)   }_{a_n} f(x_1, ..., x_n)\Delta x_1 \cdot \cdot \cdot \Delta x_n,$ 
   $ \int^{x_1}_{a_1}\cdot \cdot \cdot\int^{x_n}_{a_n} f(t_1, ..., t_n)\Delta t_1 \cdot \cdot \cdot \Delta t_n,$ 
$ \int^{ \rho^{\lambda_1-1}_1(b_1)   }_{x_1}\cdot \cdot \cdot\int^{\rho^{\lambda_n-1}_n(b_n)    }_{x_n} f(t_1, ..., t_n)\Delta t_1 \cdot \cdot \cdot \Delta t_n$, and
 $(n-1)-$fold integral $ \int^{b_2}_{a_2}\cdot \cdot \cdot\int^{b_n}_{a_n} f(x_2, ..., x_n)\Delta x_2\cdot \cdot \cdot \Delta x_n$,
respectively.

 Let $f: \Lambda^n \to \mathbb{R}$. The $partial$  $delta$ $derivative$  of $f$ with respect to $x_j \in \mathbb{T}^\kappa_j$ is defined as the limit
\[\lim_{\underset{ t_j \neq \sigma_j(x_j)}{t_j \to x_j}}\frac{f(x_1, ..., \sigma_j(x_j), ..., x_n) -f(x_1, ..., t_j, ..., x_n) }{\sigma_j(x_j)-t_j}  \]
provided that this limit exists as a finite number, and is denoted by $ \frac{\partial f(\boldsymbol{x})}{\Delta_j x_j}. $
If $f$ has partial derivatives  $\frac{\partial f(\boldsymbol{x})}{\Delta_jx_j}, j \in [1,n]_{\mathbb{N}},$  then we can also consider their partial delta  derivatives. These are called $second$ $order$ partial delta derivatives. We write
\[ 
\frac{\partial^2 f(\boldsymbol{x})}{\Delta_j x^2_j}=\frac{\partial }{\Delta_j x_j}\bigg(\frac{\partial f(\boldsymbol{x})}{\Delta_j x_j}\bigg), \quad \frac{\partial^2 f(\boldsymbol{x})}{\Delta_j x_j \Delta_i x_i} =\frac{\partial }{\Delta_j x_j}\bigg(\frac{\partial f(\boldsymbol{x})}{\Delta_i x_i}\bigg). 
 \]
Higher order partial delta derivatives are similarly defined.

 Let $\boldsymbol{\lambda}\geq \boldsymbol{1}$ be a multi-index,  we denote by  $|\boldsymbol{\lambda}| = \lambda_1+ \cdot \cdot \cdot+ \lambda_n$;  then we set
\[  \frac{\partial^{\boldsymbol{\lambda}}f(\boldsymbol{x}) }{\Delta {\boldsymbol{x}}^{\boldsymbol{\lambda}}}=\frac{\partial^{|\boldsymbol{\lambda}|}f(\boldsymbol{x}) }{\Delta_1x^{\lambda_1}_1 \cdot \cdot \cdot\Delta_nx^{\lambda_n}_n }.\]
 By $C^{n\boldsymbol{\lambda}}_{\text{rd}}(\Omega),$ we denote the set of all functions $f: \Omega \to \mathbb{R}$  which have rd-continuous derivatives    $\frac{\partial ^{k_1+ \cdot \cdot \cdot+ k_j}f(\boldsymbol{x})}{\Delta_1 x^{k_1}_1 \cdot \cdot \cdot \Delta_j x^{k_j}_j }$   for $k_j \in [1, \lambda_j]_{\mathbb{N}}$,  $j \in [1, n]_{\mathbb{N}}$.   
A function $\tau : \Omega \to \mathbb{R}$ is said to be a weight on $\Omega$ if $\tau$ is positive-valued and rd-continuous on $\Omega.$ Let us  denote by $\mathcal{W}(\Omega)$ the set of all weights on $\Omega$. 
Let $p\geq1$ and $\tau \in \mathcal{W}(\Omega)$. We represent by $\mathcal{L}^p_{\boldsymbol{a}}(\Omega, \tau, \boldsymbol{\lambda})$
the set of all functions $f: \Omega \to \mathbb{R}$  of class $C^{n\boldsymbol{\lambda}}_{\text{rd}}(\Omega)$ for which
 $\frac{\partial^{k_j}  f(\boldsymbol{x})}{\Delta_jx^{k_j}_j}|_{x_j =a_j} = 0$ for $k_j \in [0, \lambda_j-1]_{\mathbb{N}},  j \in[1, n]_{\mathbb{N}},$ and that $\int_{\Omega}|\frac{\partial^{\boldsymbol{\lambda}} f(\boldsymbol{x}) }{\Delta \boldsymbol{x}^{\boldsymbol{\lambda}}}|^{p}\tau(\boldsymbol{x})\Delta \boldsymbol{x} <\infty.$

We set \[H_{\boldsymbol{\lambda}}(\boldsymbol{x},\boldsymbol{t}) = \prod^n_{j=1}{h^{(j)}_{\lambda_j-1}(x_j,\sigma_j(t_j))}, \quad \boldsymbol{x}, \boldsymbol{t} \in \Omega,\]
where $h^{(j)}_{k}: \mathbb{T}^2_j \to \mathbb{R},$ $k \in \mathbb{N}_0,$ is such that
$h^{(j)}_{0}(t,s)\equiv 1$ for all $t, s\in \mathbb T_j,$ and $h^{(j)}_{k+1}(t,s)=\int_s^t h^{(j)}_{k}(u,s)\Delta u $ for all $ t, s\in \mathbb T_j, k \in \mathbb{N}_0. $

Let $m$ be a positive integer and $0 <R\leq \infty$. We represent by $\mathcal{H}^m_R$ the set of all functions $F : (-R, R)^m \to \mathbb{R}$ such that 
\begin{enumerate}
\item $F\in C^1((-R, R)^m)$, 
\item   $F(0, ..., 0) = 0$, and
\item   $D_iF$ for $i \in [1, m]_{\mathbb{N}}$, are non-negative and increasing in each variable on $(0, R)$, 
where $D_i = \partial/\partial{t_i},    i \in [1, m]_{\mathbb{N}}$ for all $(t_1, ..., t_m) \in (-R, R)^m.$
\end{enumerate}
We give a preliminary lemma that we shall use in Section $3$.
\begin{lemma}\label{chainp}
Let $F \in \mathcal{H}^m_R$ and  $g_i: \mathbb{T} \to [0, R)$ for $i \in [1, m]_{\mathbb{N}},$ are delta differentiable on $\mathbb{T}^{\kappa}$ such that $g^\Delta_i$ are non-negative on $\mathbb{T}^{\kappa}$;  then  the composite function $ F(g_1(x), ..., g_m(x))$ is delta differentiable on $ \mathbb{T}^\kappa$ such that 
\begin{equation}\label{chainnew}
[F(g_1(x), ..., g_m(x))]^\Delta \geq \sum^m_{i=1}{D_i F(g_1(x), ..., g_m(x))g_i^{\Delta}(x)}, \quad x \in \mathbb{T}^{\kappa}.
\end{equation}
 \end{lemma}
\begin{proof}
Fix $x \in \mathbb{T}^\kappa$ and put $ t_i = g_i(x)$,  $t^\sigma_i = g_i(\sigma(x))$, we see that $t_i \leq t^\sigma_i$ for $i \in [1, m]_{\mathbb{N}}.$  We have two cases.

 \textit{Case 1.} Suppose that  $x <\sigma(x).$ Then
\[ 
[F(t_1, ..., t_m)]^\Delta = \frac{F(t^\sigma_1, ..., t^\sigma_m)-F(t_1, ..., t_m)}{\sigma(x)-x}.\]
For all $i \in [1, m]_{\mathbb{N}}$, we set
\[ A_i= 
 \begin{cases}
  \dfrac{F(t_1, ..., t^\sigma_i, ...,  t^\sigma_m)-F(t_1, ..., t_i, t_{i+1}^\sigma, ...,  t_m^\sigma)}{t^\sigma_i-t_i}  & \text{if}\quad t_i < t^\sigma_i,    \\
  0& \text{if}\quad   t_i = t^\sigma_i.
\end{cases}
\]
Then,
\[ [F(t_1, ..., t_m)]^\Delta=\sum^m_{i=1} A_i \frac{g_i(\sigma(x))-g_i(x)}{\sigma(x) -x} = \sum^m_{i=1} A_i g^\Delta_i(x).
\]
If $t_i = t^\sigma_i$, then $g^\Delta_i(x) = 0$;   if $t_j < t^\sigma_j$ with $j \neq i$, then
\[ A_j=\frac{F(t_1, ..., t^\sigma_j, ...,  t^\sigma_m)-F(t_1, ..., t_j, t_{j+1}^\sigma, ...,  t_m^\sigma)}{t^\sigma_j-t_j}  =D_jF(t_1, ..., c_j, t^\sigma_{j+1}, ..., t^\sigma_m),  \]
by the mean value theorem, where $c_j \in (t_j, t^\sigma_j)$. Since $D_jF$, $i \in [1, m]_{\mathbb{N}}$, are increasing in each variable on $(0,R)$, we have
\[D_jF(t_1, ..., c_j, t^\sigma_{j+1}, ..., t^\sigma_{m}) \geq D_jF(t_1, ..., t_m) = D_j F(g_1(x), ..., g_m(x)),\]
which yields \eqref{chainnew}.

\textit{ Case 2.} Suppose that $x = \sigma(x)$. We set
\[ F_i(g(s)) :=F(g_1(x), ..., g_{i-1}(x), g_i(s), g_{i+1}(s), ..., g_m(s)),   \]
\[ F_i(g(x)) :=F(g_1(x), ..., g_{i-1}(x), g_i(x), g_{i+1}(s), ..., g_m(s)), \quad i \in [1, m]_{\mathbb{N}}.   \]
We have
\[ [F(g_1(x), ..., g_m(x))]^\Delta = \lim_{s \to x}\frac{F(g_1(s), ..., g_m(s))-F(g_1(x), ..., g_m(x))}{s-x},  \]
where 
\[\frac{F(g_1(s), ..., g_m(s))-F(g_1(x), ..., g_m(x))}{s-x} = \sum^m_{i =1} {\frac{g_i(s)-g_i(x)}{s-x} \frac{F_i(g(s)) -F_i(g(x))}{g_i(s)-g_i(x)} }.  \]
By the mean value theorem, there exist $\xi_i(s), i \in [1, m]_{\mathbb{N}}$, which are  between  $g_i(x)$ and  $g_i(s)$, such that
\[\frac{F_i(g(s)) -F_i(g(x))}{g_i(s)-g_i(x)}   = D_iF(g_1(x), ..., g_{i-1}(x), \xi_i(s), g_{i+1}(s), ..., g_m(s)).\]
Since $g_i$ for $i \in [1, m]_{\mathbb{N}},$ are delta differentiable on $\mathbb{T}^\kappa$, then $g_i$ for $i \in [1, m]_{\mathbb{N}},$ are continuous at $x$. Therefore,  $\lim_{s \to x}{\xi_i(s)} = g_i(x)$ and $\lim_{s \to x}{g_i(s)} = g_i(x)$ for $i \in [1, m]_{\mathbb{N}},$ which gives us the desired result. \qed
\end{proof}

\section{Integral inequalities on time scales}
\begin{theorem}\label{THEO1}
 Let $ F \in \mathcal{H}^m_R$ for $0 <R\leq \infty$, and  $f_i : \Omega \to (-R, R)$  which satisfies $\int_{\Omega}|\frac{\partial^{  \boldsymbol{1}    }f_i(\textbf{x})}{\Delta \textbf{x}^{   \boldsymbol{1}  }}|\Delta \textbf{x}<R$ for  $i \in [1, m]_{\mathbb{N}}$. If  $f_i \in \mathcal{L}^1_{\textbf{a}}(\Omega, 1, \boldsymbol{1})$ for all $i \in [1,m]_{\mathbb{N}},$ then
\begin{equation}\label{ptTHEO1}
\begin{split}
&\int_{\Omega}\bigg(\sum^m_{i=1}{D_iF(|f_1(\textbf{x})|, ..., |f_m(\textbf{x})|)\bigg|\frac{\partial^{\boldsymbol{1}}f_i(\textbf{x})}{\Delta \textbf{x}^{  \boldsymbol{1}   } }\bigg|}\bigg) \Delta \textbf{x} \\
&\leq  F\bigg(\int_{\Omega}\bigg|\frac{\partial^{  \boldsymbol{1}  }f_1(\textbf{x})}{\Delta \textbf{x}^{\boldsymbol{1}   } }\bigg|\Delta \textbf{x}, ..., \int_{\Omega}\bigg|\frac{\partial^{  \boldsymbol{1}    }f_m(\textbf{x})}{\Delta \textbf{x}^{   \boldsymbol{1}  }}\bigg|\Delta \textbf{x}  \bigg). \\
\end{split}
\end{equation}
\end{theorem}

\begin{proof}
Since $f_i \in \mathcal{L}^1_{\boldsymbol{a}}(\Omega, 1, \boldsymbol{1})$, it follows that
\begin{equation}\label{eqTHEO1}
f_i(\boldsymbol{x})  = \int_{\Omega_{\boldsymbol{x}}}\frac{\partial^{\boldsymbol{1}}f_i(\boldsymbol{t})}{\Delta \boldsymbol{t}^{\boldsymbol{1}} } \Delta \boldsymbol{t}
\end{equation}
for all $\boldsymbol{x} \in \Omega$ and all $i \in [1, m]_{\mathbb{N}}$.
Let 
\[ g_i(x_1): =\int^{x_1}_{a_1}\int_{\Omega'}\bigg|\frac{\partial^{\boldsymbol{1}}f_i(\boldsymbol{t})}{\Delta \boldsymbol{t}^{\boldsymbol{1}}}\bigg| \Delta \boldsymbol{t} \quad \text{for} \quad x_1 \in [a_1, b_1]_{\mathbb{T}}, \quad  i \in [1, m]_{\mathbb{N}}.  \]
We see that  $|f_i(\boldsymbol{x})| \leq g_i(x_1)$ for $\boldsymbol{x} \in \Omega$ and  functions $g_i, i \in [1, m]_{\mathbb{N}},$ are increasing on $[a_1, b_1]_{\mathbb{T}}$. By  $F \in \mathcal{H}^m_R$, we obtain
\[
\begin{split}
&\int_{\Omega}\bigg(\sum^m_{i=1}{D_iF(|f_1(\boldsymbol{x})|, ..., |f_m(\boldsymbol{x})|)\bigg|\frac{\partial^{\boldsymbol{1}}f_i(\boldsymbol{x})}{\Delta \boldsymbol{x}^{\boldsymbol{1}}}\bigg|}\bigg) \Delta \boldsymbol{x}\\
& \leq \int_{\Omega}\bigg(\sum^m_{i=1}{D_iF(g_1(x_1)), ..., g_m(x_1))\bigg|\frac{\partial^{\boldsymbol{1}}f_i(\boldsymbol{x})}{\Delta \boldsymbol{x}^{\boldsymbol{1}}}\bigg|}\bigg) \Delta \boldsymbol{x}\\
& \leq \int^{b_1}_{a_1}\bigg(\sum^m_{i=1}{D_iF(g_1(x_1), ..., g_m(x_1))\int_{\Omega'}{\bigg|\frac{\partial^{\boldsymbol{1}}f_i(\boldsymbol{x})}{\Delta \boldsymbol{x}^{\boldsymbol{1}}}\bigg|}\Delta \boldsymbol{x'}}\bigg) \Delta x_1\\
& \leq \int^{b_1}_{a_1}\bigg(\sum^m_{i=1}{D_iF(g_1(x_1), ..., g_m(x_1))\frac{\partial g_i(x_1)}{\Delta_1x_1}\bigg)} \Delta x_1,\\
\end{split}
\]
which, in view of Lemma \ref{chainp}, yields
\[
\begin{split}
&\int_{\Omega}\bigg(\sum^m_{i=1}{D_iF(|f_1(\boldsymbol{x})|, ..., |f_m(\boldsymbol{x})|)\bigg|\frac{\partial^{\boldsymbol{1}}f_i(\boldsymbol{x})}{\Delta \boldsymbol{x}^{\boldsymbol{1}}}\bigg|}\bigg) \Delta \boldsymbol{x}  \\
& \leq \int^{b_1}_{a_1}F^{\Delta_1}(g_1(x_1), ..., g_m(x_1))\Delta x_1\\
& =F\bigg(\int_{\Omega}\bigg|\frac{\partial^{\boldsymbol{1}}f_1(\boldsymbol{x})}{\Delta \boldsymbol{x}^{\boldsymbol{1}}}\bigg|\Delta \boldsymbol{x}, ..., \int_{\Omega}\bigg|\frac{\partial^{\boldsymbol{1}}f_m(\boldsymbol{x})}{\Delta \boldsymbol{x}^{\boldsymbol{1}}}\bigg|\Delta \boldsymbol{x}  \bigg), \\
\end{split}
\]
which completes the proof. \qed
\end{proof}

\begin{remark} From Theorem \ref{THEO1} we can obtain many known results.
\begin{enumerate}
\item If $\mathbb{T} = \mathbb{R}$, then Theorem \ref{THEO1} becomes \cite[Theorem 1]{BPe}  which was established by Brneti\'{c} and Pe\u{c}ari\'{c}.
\item If $\mathbb{T} = \mathbb{R}$, $n = m = 1,$ and   $F$ is convex on $[0, \infty)$, then inequality \eqref{ptTHEO1} reduces to \eqref{eq1.2}.
\item Let $\mathbb{T} = \mathbb{R}$,  $n=1$, and $F(x_1, ..., x_m) = |x_1 \cdot \cdot \cdot x_m|$; then inequality \eqref{ptTHEO1}  becomes \cite[Theorem 1]{pach}.
\end{enumerate}
\end{remark}

The following theorem is a generalization of Theorem \ref{THEO1}.
\begin{theorem}\label{THEO3}
 Let $ F \in \mathcal{H}^m_R$ for $0 <R\leq \infty$, and  $f_i : \Omega \to (-R, R)$ which satisfies  $  H_{\boldsymbol{\lambda}}(\textbf{b},\textbf{a})\int_{\Omega}|\frac{\partial^{  \boldsymbol{\lambda}    }f_i(\textbf{x})}{\Delta \textbf{x}^{   \boldsymbol{1}  }}|\Delta \textbf{x}<R$  for $i \in [1, m]_{\mathbb{N}}$. If  $f_i \in \mathcal{L}^1_{\textbf{a}}(\Omega, 1, \boldsymbol{\lambda})$ for all $i \in [1,m]_{\mathbb{N}},$ then
\begin{equation}\label{ptTHEO3}
\begin{split}
&\int_{\Omega}\bigg(\sum^m_{i=1}{D_iF(|f_1(\textbf{x})|, ..., |f_m(\textbf{x})|)\bigg|\frac{\partial^{\boldsymbol{\lambda}}f_i(\boldsymbol{x})}{\Delta \textbf{x}^{\boldsymbol{\lambda}}}\bigg|}\bigg) \Delta \textbf{x} \\
&\leq \frac{1}{H_{\boldsymbol{\lambda}}(\textbf{b},\textbf{a})}F\bigg(H_{\boldsymbol{\lambda}}(\textbf{b},\textbf{a})\int_{\Omega} \bigg|\frac{\partial^{\boldsymbol{\lambda}}f_1(\textbf{x}) }{\Delta \textbf{x}^{\boldsymbol{\lambda}}}\bigg|\Delta \textbf{x}, ...,H_{\boldsymbol{\lambda}}(\textbf{b}, \textbf{a})\int_{\Omega}\bigg|\frac{\partial^{\boldsymbol{\lambda}} f_m(\textbf{x})}{\Delta \textbf{x}^{\boldsymbol{\lambda}}}\bigg|\Delta \textbf{x}  \bigg).  \\
\end{split}
\end{equation}
\end{theorem}
\begin{proof}
By $f_i \in  \mathcal{L}^1_{\boldsymbol{a}}(\Omega, 1, \boldsymbol{\lambda}), i \in [1,m]_{\mathbb{N}},$ and Taylor's formula \cite{Hig}, we have
\[
f_i(\boldsymbol{x}) = \int_{\Omega_{\boldsymbol{x}}} H_{\boldsymbol{\lambda}}(\boldsymbol{x},\boldsymbol{t})\frac{\partial^{\boldsymbol{\lambda}}  f_i(\boldsymbol{t})}{\Delta \boldsymbol{t}^{\boldsymbol{\lambda}}} \Delta \boldsymbol{t}, \quad \boldsymbol{x} \in \Omega.
\]
Therefore, 
\[
|f_i(\boldsymbol{x})| \leq \int_{\Omega_{\boldsymbol{x}}} H_{\boldsymbol{\lambda}}(\boldsymbol{x},\boldsymbol{t})\bigg|\frac{\partial^{\boldsymbol{\lambda}}f_i(\boldsymbol{t}) }{\Delta \boldsymbol{t}^{\boldsymbol{\lambda}}}\bigg| \Delta \boldsymbol{t} \leq  H_{\boldsymbol{\lambda}}(\boldsymbol{b},\boldsymbol{a}) \int_{\Omega_{\boldsymbol{x}}}\bigg|\frac{\partial^{\boldsymbol{\lambda}}f_i(\boldsymbol{t}) }{\Delta \boldsymbol{t}^{\boldsymbol{\lambda}}}\bigg| \Delta \boldsymbol{t}, \quad \boldsymbol{x} \in \Omega.
\]
Now, we define
 \[  u_i(\boldsymbol{x}): =H_{\boldsymbol{\lambda}}(\boldsymbol{b},\boldsymbol{a}) \int_{\Omega_{\boldsymbol{x}}}\bigg|\frac{\partial^{\boldsymbol{\lambda}}f_i(\boldsymbol{t}) }{\Delta \boldsymbol{t}^{\boldsymbol{\lambda}}}\bigg| \Delta \boldsymbol{t}, \quad \boldsymbol{x}\in \Omega, \quad i \in [1, m]_{\mathbb{N}},  \]
we see that $ |f_i(\boldsymbol{x})| \leq u_i(\boldsymbol{x})$ for $\boldsymbol{x} \in \Omega$ and $u_i(\boldsymbol{x})$ are increasing in each variable and
\[
\bigg|\frac{\partial^{\boldsymbol{\lambda}}f_i(\boldsymbol{x}) }{\Delta \boldsymbol{x}^{\boldsymbol{\lambda}}}\bigg| = \frac{1}{H_{\boldsymbol{\lambda}}(\boldsymbol{b},\boldsymbol{a})} \frac{\partial^{\boldsymbol{1}}u_i(\boldsymbol{x}) }{\Delta \boldsymbol{x}^{\boldsymbol{1}}}, \quad \boldsymbol{x} \in \Omega,  \quad i \in [1, m]_{\mathbb{N}}.
\]

Since $D_iF$ for  $  i \in [1, m]_{\mathbb{N}}$, are non-negative, continuous and increasing in each variable on $(0, R)$, we have
\[
\begin{split}
&\int_{\Omega}\bigg(\sum^m_{i=1}{D_iF(|f_1(\boldsymbol{x})|, ..., |f_m(\boldsymbol{x})|)\bigg|\frac{\partial^{\boldsymbol{\lambda}}f_i(\boldsymbol{x}) }{\Delta \boldsymbol{x}^{\boldsymbol{\lambda}}}\bigg|}\bigg) \Delta \boldsymbol{x}\\
&\leq   \frac{1}{H_{\boldsymbol{\lambda}}(\boldsymbol{b},\boldsymbol{a})} \int_{\Omega}\bigg(\sum^m_{i=1}{D_iF(u_1(\boldsymbol{x}), ..., u_m(\boldsymbol{x}))\frac{\partial^{\boldsymbol{1}} u_i(\boldsymbol{x})}{\Delta \boldsymbol{x}^{\boldsymbol{1}}}}\bigg) \Delta \boldsymbol{x}, \\
\end{split}
\]
which, in view of \eqref{ptTHEO1}, gives \eqref{ptTHEO3}. \qed
\end{proof}
\begin{remark} 
Let $\mathbb{T} = \mathbb{R}$; then Theorem \ref{THEO3} becomes  \cite[Theorem  2.1]{AND} which was established by  Andri\'{c}.
 \end{remark}

Next, for $0<R\le \infty$, we denote by $\mathcal{G}^{1, m}_R$ the class of all functions $G: (-R, R)^m \to  \mathbb{R}$  satisfying the following conditions:
\begin{enumerate}
\item[(i)] $G \in C^1((-R, R)^m)$,
\item[(ii)] $G(0, ..., 0)=0$, and
\item[(iii)] if $x_i\leq y_i^{1/p}z_i^{1/q},  0<x_i, y_i, z_i<R$ for $ i \in [1, m]_{\mathbb{N}},$ then
 $0\leq D_iG(x_1, ..., x_m)\leq [D_iG(y_1, ..., y_m)]^{1/p}[D_iG(z_1, ..., z_m)]^{1/q}$,  where  $p$, $q$ are conjugate exponents
$1/p + 1/q = 1$.
\end{enumerate}

\begin{remark}
If $G \in \mathcal{G}^{1, m}_R$, then $G \in \mathcal{H}^m_R.$
\end{remark}
\begin{proof}
Let $G \in \mathcal{G}^{1, m}_R$. For each $i \in [1, m]_{\mathbb{N}}$ and  $0 <x_i \leq y_i <R,$ from $(iii)$ we have
\[
\begin{split} 
 0\leq D_iG(x_1, ..., x_i, ..., x_m)&\leq [D_iG(y_1, ..., y_i, ...,  y_m)]^{1/p}[D_iG(y_1, ..., y_i, ...,  y_m)]^{1/q}\\
&=D_iG(y_1, ..., y_i, ...,  y_m).  \\
\end{split}
 \] 
Therefore, $G \in \mathcal{H}^m_R.$
\end{proof}

\begin{example}
The functions $G(x_1, ..., x_m) = |x_1|^{\gamma_1}\sign(x_1)\cdot \cdot \cdot|x_m|^{\gamma_m}\sign(x_m)$,  $H(x_1, ..., x_m) = |x_1|^{\gamma_1}+ \cdot \cdot \cdot +|x_m|^{\gamma_m}$ for $\gamma_i \geq 1$ for all $i \in [1, m]_{\mathbb{N}}$ are in $\mathcal{G}^{1, m}_{\infty}$.
\end{example}
From on now, we always assume that $\alpha, \beta>0$ and $\alpha + \beta >1$ and $G \in \mathcal{G}^{1, m}_R$. We have the following result.
\begin{theorem}\label{THEO4}
Let $\omega_i, \tau_i, \in \mathcal{W}(\Omega)$, and $f_i : \Omega \to (-R, R)$ be such that 
\[\int_{\Omega}\bigg|\frac{\partial^{\boldsymbol{\lambda}} f_i(\textbf{x}) }{\Delta \textbf{x}^{\boldsymbol{\lambda}}}\bigg|^{\alpha+\beta}\tau_i(\textbf{x})\Delta \textbf{x}<R \quad \text{for} \quad  i \in [1,m]_{\mathbb{N}}. \]

If $f_i \in \mathcal{L}^{\alpha+\beta}_{\textbf{a}}(\Omega, \tau_i, \boldsymbol{\lambda})$ for all  $ i \in [1, m]_{\mathbb{N}}$ and
\begin{equation}\label{k2}
K_{\Omega}:=\bigg[\int_{\Omega}\bigg(\sum^m_{i=1}{[D_iG(V_1(\textbf{x}), ..., V_m(\textbf{x}))]^{{\frac{\alpha(\alpha+\beta-1)}{\beta}}}\omega^{\frac{\alpha+\beta}{\beta}}_i(\textbf{x})\tau^{-\frac{\alpha}{\beta}}_i(\textbf{x})  }\bigg)\Delta \textbf{x}\bigg]^{\frac{\beta}{\alpha+\beta}}<\infty,
\end{equation}
where
\[V_i(\textbf{x}): =\int_{\Omega_\textbf{x}} \bigg(H_{\boldsymbol{\lambda}}(\textbf{x},\textbf{t})\bigg)^{\frac{\alpha+\beta}{\alpha+\beta-1}}(\tau_i(\textbf{t}))^{\frac{1}{1-\alpha-\beta}}\Delta \textbf{t} \]
for  $\textbf{ x }\in \Omega,  i \in [1, m]_{\mathbb{N}}$, then
\begin{equation}\label{ptTHEO4}
\begin{split}
&\int_{\Omega}\bigg(\sum^m_{i=1}{[D_iG(|f_1(\textbf{x})|, ..., |f_m(\textbf{x})|)]^\alpha\bigg|\frac{\partial^{\boldsymbol{\lambda}}f_i(\textbf{x}) }{\Delta \textbf{x}^{\boldsymbol\lambda}}\bigg|^\alpha\omega_i(\textbf{x})}\bigg) \Delta \textbf{x}   \\
&\leq K_{\Omega} \bigg[G\bigg(\int_{\Omega}\bigg|\frac{\partial^{\boldsymbol{\lambda}} f_1(\textbf{x}) }{\Delta \textbf{x}^{\boldsymbol{\lambda}}}\bigg|^{\alpha+\beta}\tau_1(\textbf{x})\Delta \textbf{x}, ..., \int_{\Omega}\bigg|\frac{\partial^{\boldsymbol{\lambda}}f_m(\textbf{x})  }{\Delta \textbf{x}^{\boldsymbol{\lambda}}}\bigg|^{\alpha+\beta}\tau_m(\textbf{x})\Delta \textbf{x}  \bigg)\bigg]^{\frac{\alpha}{\alpha+\beta}}. \\
\end{split}
\end{equation}
\end{theorem}

\begin{proof}
As in the proof of Theorem \ref{THEO3}, for any $\boldsymbol{x} \in \Omega$ and all $i \in [1, m]_{\mathbb{N}}$, we have
\begin{equation}\label{eq36}
|f_i(\boldsymbol{x})| \leq \int_{\Omega_{\boldsymbol{x}}} H_{\boldsymbol{\lambda}}(\boldsymbol{x},\boldsymbol{t})\bigg|\frac{\partial^{\boldsymbol{\lambda}} f_i(\boldsymbol{t})}{\Delta \boldsymbol{t}^{\boldsymbol{\lambda}}}\bigg| \Delta \boldsymbol{t}.
\end{equation}
Applying H\"{o}lder's inequality with  indices $(\alpha+\beta)/(\alpha+\beta-1)$ and $(\alpha+\beta)$ to \eqref{eq36}, we get
\begin{align}
\notag|f_i(\boldsymbol{x})| &\leq  \bigg(\int_{\Omega_{\boldsymbol{x}}} \bigg(H_{\boldsymbol{\lambda}}(\boldsymbol{x},\boldsymbol{t})\bigg)^{\frac{\alpha+\beta}{\alpha+\beta-1}}(\tau_i(\boldsymbol{t}))^{\frac{1}{1-\alpha-\beta}} \Delta \boldsymbol{t}\bigg)^{\frac{\alpha+\beta-1}{\alpha+\beta}} \\
\notag&\quad \times  \bigg(\int_{\Omega_{\boldsymbol{x}}}\bigg|\frac{\partial^{\boldsymbol{\lambda}}f_i(\boldsymbol{t})  }{\Delta \boldsymbol{t}^{\boldsymbol{\lambda}}}\bigg|^{\alpha+\beta}  \tau_i(\boldsymbol{t}) \Delta t\bigg)^{\frac{1}{\alpha+\beta}}\\
 &=: \bigg(V_i(\boldsymbol{x})\bigg)^{\frac{\alpha+\beta-1}{\alpha+\beta}} \bigg(U_i(\boldsymbol{x})\bigg)^{\frac{1}{\alpha+\beta}},
\label{eqfa}
\end{align}
where
\[U_i(\boldsymbol{x}) =\int_{\Omega_{\boldsymbol{x}}}\bigg|\frac{\partial^{\boldsymbol{\lambda}}f_i(\boldsymbol{t}) }{\Delta \boldsymbol{t}^{\boldsymbol{\lambda}}}\bigg|^{\alpha+\beta}\tau_i(\boldsymbol{t}) \Delta \boldsymbol{t},  \quad \boldsymbol{x} \in \Omega, \quad  i \in [1, m]_{\mathbb{N}}.\]
Thus, since $G\in \mathcal{G}^{1, m}_R$ from \eqref{eqfa}  we obtain
\[ 
\begin{split}
D_iG(|f_1(\boldsymbol{x})|, ..., |f_m(\boldsymbol{x})|) \leq  [&D_iG(V_1(\boldsymbol{x}), ..., V_m(\boldsymbol{x}) )]^{\frac{\alpha+\beta-1}{\alpha+\beta}}\\
& \times [D_iG(U_1(\boldsymbol{x}), ..., U_m(\boldsymbol{x}))]^{\frac{1}{\alpha+\beta}};\\
\end{split}\]
hence
\begin{equation}\label{eq37}
\begin{split}
&\sum^m_{i=1}{[D_iG(|f_1(\boldsymbol{x})|, ..., |f_m(\boldsymbol{x})|)]^\alpha\bigg|\frac{\partial^{\boldsymbol{\lambda}} f_i(\boldsymbol{x}) }{\Delta \boldsymbol{x}^{\boldsymbol{\lambda}}}\bigg|^\alpha\omega_i(\boldsymbol{x})}\\
& \leq \sum^m_{i=1}[D_iG(V_1(\boldsymbol{x}), ..., V_m(\boldsymbol{x}))]^{\frac{\alpha(\alpha+\beta-1)}{\alpha+\beta}}\\
& \quad  \times [D_iG(U_1(\boldsymbol{x}), ..., U_m(\boldsymbol{x}))]^{\frac{\alpha}{\alpha+\beta}}\bigg|\frac{\partial^{\boldsymbol{\lambda}} f_i(\boldsymbol{x}) }{\Delta \boldsymbol{x}^{\boldsymbol{\lambda}}}\bigg|^\alpha\omega_i(\boldsymbol{x}).\\
\end{split}
\end{equation}
By applying H\"{o}lder's inequality for sum with  indices $(\alpha+\beta)/\beta$ and $(\alpha+\beta)/\alpha$ to \eqref{eq37}, we get
\begin{equation}\label{eqsum}
\begin{split}
&\sum^m_{i=1}{[D_iG(|f_1(\boldsymbol{x})|, ..., |f_m(\boldsymbol{x})|)]^\alpha\bigg|\frac{\partial^{\boldsymbol{\lambda}} f_i(\boldsymbol{x})}{\Delta \boldsymbol{x}^{\boldsymbol{\lambda}}}\bigg|^\alpha}\omega_i(\boldsymbol{x})\\
 &\leq \bigg(\sum^m_{i=1}{[D_iG(V_1(\boldsymbol{x}), ..., V_m(\boldsymbol{x})) ]^{ \frac{\alpha(\alpha+\beta-1)}{\beta}    }\omega^{\frac{\alpha+\beta}{\beta}}_i(\boldsymbol{x})\tau^{-\frac{\alpha}{\beta}}_i(\boldsymbol{x})}\bigg)^{\frac{\beta}{\alpha+\beta}}\\
&\quad \times\bigg(\sum^m_{i=1}{D_iG(U_1(\boldsymbol{x}), ..., U_m(\boldsymbol{x}))\bigg|\frac{\partial^{\boldsymbol{\lambda}}f_i(\boldsymbol{x}) }{\Delta \boldsymbol{x}^{\boldsymbol{\lambda}}}\bigg|^{\alpha+\beta}}\tau_i(\boldsymbol{x})  \bigg)^{{\frac{\alpha}{\alpha+\beta}}}.\\
\end{split}
\end{equation}
Integrating both sides of \eqref{eqsum} with respect to $\boldsymbol{x}$ over $\Omega$ and using H\"{o}lder's inequality  with  indices $(\alpha+\beta)/\beta$ and $(\alpha+\beta)/\alpha$ give
\begin{equation}
\begin{split}
&\int_{\Omega}\bigg(\sum^m_{i=1}{[D_iG(|f_1(\boldsymbol{x})|, ..., |f_m(\boldsymbol{x})|)]^\alpha\bigg|\frac{\partial^{\boldsymbol{\lambda}} f_i(\boldsymbol{x})}{\Delta \boldsymbol{x}^{\boldsymbol{\lambda}}}\bigg|^\alpha}\omega_i(\boldsymbol{x})\bigg)\Delta \boldsymbol{x}\\
&\leq  \bigg[\int_{\Omega}\bigg(\sum^m_{i=1}{[D_iG(V_1(\boldsymbol{x}), ..., V_m(\boldsymbol{x}))]^{ \frac{\alpha(\alpha+\beta-1)}{\beta}  } \omega^{\frac{\alpha+\beta}{\beta}}_i(\boldsymbol{x})\tau^{-\frac{\alpha}{\beta}}_i(\boldsymbol{x})  }\bigg)\Delta \boldsymbol{x}\bigg]^{\frac{\beta}{\alpha+\beta}}\\
&\quad \times \bigg[ \int_{\Omega}\bigg(\sum^m_{i=1}{D_iG(U_1(\boldsymbol{x}), ..., U_m(\boldsymbol{x}))\frac{\partial^{\boldsymbol{1}} U_i(\boldsymbol{x}) }{\Delta \boldsymbol{x}^{\boldsymbol{1}}}}\bigg) \Delta \boldsymbol{x}   \bigg]^{\frac{\alpha}{\alpha+\beta}}, \\
\end{split}
\end{equation}
which, in view of \eqref{ptTHEO1}, yields \eqref{ptTHEO4}. This concludes the proof. \qed
\end{proof}
\begin{remark}\label{Re1}
\begin{enumerate}
\item In the special case when  $n=m=1, G(x) =|x|^{(\alpha+\beta)/\alpha}$, $ \omega_1= \tau_1\equiv 1$, $\boldsymbol{\lambda} =\lambda_1,$ then inequality \eqref{ptTHEO4} becomes  \cite[Theorem 3.2]{Lihan}.
\item  If $n= \boldsymbol{\lambda}=\alpha = 1, m=2, G(x_1, x_2) = |x_1x_2|, \beta = p-1,$ for $1<p\leq 2,$
$\omega_1=\psi^{\sigma}$, $\tau_1=\phi [\psi^{\sigma}]^{p/2}$, where $\psi^{\sigma} = \psi \circ \sigma$, $\phi, \psi \in \mathcal{W}([a,b]_{\mathbb T})$ and $\psi$ is decreasing  on $[a,b]_{\mathbb T}$, then inequality \eqref{ptTHEO4} reduces to
\[
\begin{split}
&\int_a^b \psi^{\sigma}(x)(|f_1^{\Delta}(x)f_2(x)|+|f_1(x)f_2^{\Delta}(x)|)\Delta x\\
& \le \frac{K(a, b)}{2^{2/p}}\left[\int_a^b\left(|f_1^{\Delta}(x)|^{p}+|f_2^{\Delta}(x)|^{p}\right)\phi(x) [\psi^{\sigma}(x)]^{p/2}\Delta x \right]^{2/p},\\
\end{split}
\]
where 
\begin{align} 
\notag K(a, b)&=\left[2\int_a^b\frac{[\psi^{\sigma}(x)]^{q/2}}{\phi^{q/p}(x)}\left(\int_a^x\frac{\Delta t}{\phi^{q/p}(t)[\psi^{\sigma}(t)]^{q/2}}\right)\Delta x \right]^{1/q}\\
 \notag& \le \left[2\int_a^b\frac{1}{\phi^{q/p}(x)}\left(\int_a^x\frac{\Delta t}{\phi^{q/p}(t)}\right)\Delta x \right]^{1/q}\\
\label{eqRe1}&\le \left[\int_a^b\frac{\Delta x}{\phi^{q/p}(x)} \right]^{2/q},
\end{align} 
where $p, q$ are conjugate exponents. Therefore, Theorem \ref{THEO4}  improves and generalizes \cite[Theorem 3.1]{YZ}. 
\item Note that when $\mathbb{T} = \mathbb{R}$,   $n =m =2, \alpha =s \geq 1, \beta = 2r+s $ for $r \geq 0$, $  G(x_1,x_2) = |x_1x_2|^{(r+s)/s}, \omega_1 = \omega_2 = \tau_1 = \tau_2=\omega,$ where $\omega$ is decreasing in each variables,  we see that inequality \eqref{ptTHEO4} improves   \cite[Theorem 1]{Che}.
\end{enumerate}
\end{remark}

From Theorem \ref{THEO4} we have the following result.
\begin{corollary}
Assume that conditions in Theorem \ref{THEO4} hold, then
\begin{equation}\label{coro4}
\begin{split}
&\int_{\Omega}\bigg(\sum^m_{i=1}{\bigg[D_iG\bigg(\bigg|\frac{\partial^{\boldsymbol{\xi}_1}f_1(\textbf{x})  }{\Delta \textbf{x}^{\boldsymbol{\xi}_1}}\bigg|, ..., \bigg|\frac{\partial^{\boldsymbol{\xi}_m}f_m(\textbf{x})  }{\Delta \textbf{x}^{\boldsymbol{\xi}_m}}\bigg|     \bigg)\bigg]^\alpha\bigg|\frac{\partial^{\boldsymbol{\lambda}}f_i(\textbf{x}) }{\Delta \textbf{x}^{\boldsymbol\lambda}}\bigg|^\alpha\omega_i(\textbf{x})}\bigg) \Delta \textbf{x}   \\
&\leq \hat{K}_{\Omega} \bigg[G\bigg(\int_{\Omega}\bigg|\frac{\partial^{\boldsymbol{\lambda}} f_1(\textbf{x}) }{\Delta \textbf{x}^{\boldsymbol{\lambda}}}\bigg|^{\alpha+\beta}\tau_1(\textbf{x})\Delta \textbf{x}, ..., \int_{\Omega}\bigg|\frac{\partial^{\boldsymbol{\lambda}}f_m(\textbf{x})  }{\Delta \textbf{x}^{\boldsymbol{\lambda}}}\bigg|^{\alpha+\beta}\tau_m(\textbf{x})\Delta \textbf{x}  \bigg)\bigg]^{\frac{\alpha}{\alpha+\beta}} \\
\end{split}
\end{equation}
for  $\boldsymbol{\xi}_i \leq \boldsymbol{\lambda}, i \in [1, n]_{\mathbb{N}}$, 
where  $\hat{K}_{\Omega} $ is obtained by substituting 
\[ V_i(\textbf{x}): =\int_{\Omega_\textbf{x}}(H_{\boldsymbol{\lambda} - \boldsymbol{\xi}_i}(\textbf{x},\textbf{t}))^{\frac{\alpha+\beta}{\alpha+\beta-1}}(\tau_i(\textbf{t}))^{\frac{1}{1-\alpha-\beta}}\Delta \textbf{t}\] into \eqref{k2}.
\end{corollary}
\begin{remark}
 Inequality \eqref{coro4} is the same as  \cite[Theorem 2.10]{SAK1}, if we take  $n =m=1, G(x) = |x|^\gamma, \gamma \geq 1$.
\end{remark}

The following theorem is similar to Theorem \ref{THEO4}.
\begin{theorem}\label{THEO4'}
Let $\omega_i, \tau_i \in \mathcal{W}(\Omega^{\kappa^{\boldsymbol{\lambda-1}}})$, and  $f_i : \Omega \to (-R, R)$ which satisfies
\[\int_{\Omega^{\kappa^{\boldsymbol{\lambda-1}}}}\bigg|\frac{\partial^{\boldsymbol{\lambda}} f_i(\textbf{x}) }{\Delta \textbf{x}^{\boldsymbol{\lambda}}}\bigg|^{\alpha+\beta}\tau_i(\textbf{x})\Delta \textbf{x}<R \quad \text{for} \quad  i \in [1,m]_{\mathbb{N}}. \]

 If  $f_i \in \mathcal{L}^{\alpha+\beta}_{{\boldsymbol \rho}^{\boldsymbol{\lambda}- \textbf{1}}(\textbf{b})}(\Omega^{\kappa^{\boldsymbol{\lambda-1}}}, \tau_i, \boldsymbol{\lambda})$ for all  $ i \in [1,m]_{\mathbb{N}}$ and
\begin{equation}\label{k2'}
K^*_{\Omega^{\kappa^{\boldsymbol{\lambda-1}}}}:=\bigg[\int_{\Omega^{\kappa^{\boldsymbol{\lambda-1}}}}\bigg(\sum^m_{i=1}{[D_iG(V^*_1(\textbf{x}), ..., V^*_m(\textbf{x}))]^{ {\frac{\alpha(\alpha+\beta-1)}{\beta}}  }\omega^{\frac{\alpha+\beta}{\beta}}_i(\textbf{x})\tau^{-\frac{\alpha}{\beta}}_i(\textbf{x})  }\bigg)\Delta \textbf{x}\bigg]^{\frac{\beta}{\alpha+\beta}}
\end{equation}
is finite, where
\[V^*_i(\textbf{x}) :=\int_{\bar{\Omega}_\textbf{x}} |H_{\boldsymbol{\lambda}}(\textbf{x},\textbf{t})|^{\frac{\alpha+\beta}{\alpha+\beta-1}}(\tau_i(\textbf{t}))^{\frac{1}{1-\alpha-\beta}}\Delta \textbf{t}   \]
for $\textbf{x} \in \Omega^{\kappa^{\boldsymbol{\lambda-1}}},  i \in [1, m]_{\mathbb{N}},$   then
\begin{equation}\label{eq4'}
\begin{split}
&\int_{\Omega^{\kappa^{\boldsymbol{\lambda-1}}}}\bigg(\sum^m_{i=1}{[D_iG(|f_1(\textbf{x})|, ..., |f_m(\textbf{x})|)]^\alpha\bigg|\frac{\partial^{\boldsymbol{\lambda}}f_i(\textbf{x}) }{\Delta \textbf{x}^{\boldsymbol\lambda}}\bigg|^\alpha\omega_i(\textbf{x})}\bigg) \Delta \textbf{x} \leq K^*_{\Omega^{\kappa^{\boldsymbol{\lambda-1}}}}  \\
&\quad \times \bigg[G\bigg(\int_{\Omega^{\kappa^{\boldsymbol{\lambda-1}}}}\bigg|\frac{\partial^{\boldsymbol{\lambda}} f_1(\textbf{x}) }{\Delta \textbf{x}^{\boldsymbol{\lambda}}}\bigg|^{\alpha+\beta}\tau_1(\textbf{x})\Delta \textbf{x}, ..., \int_{\Omega^{\kappa^{\boldsymbol{\lambda-1}}}}\bigg|\frac{\partial^{\boldsymbol{\lambda}}f_m(\textbf{x})  }{\Delta \textbf{x}^{\boldsymbol{\lambda}}}\bigg|^{\alpha+\beta}\tau_m(\textbf{x})\Delta \textbf{x}  \bigg)\bigg]^{\frac{\alpha}{\alpha+\beta}}. \\
\end{split}
\end{equation}
\end{theorem}

\begin{proof}
For all $ i \in[1, m]_{\mathbb{N}},$ and $\boldsymbol{x} \in \Omega^{\kappa^{\boldsymbol{\lambda-1}}}$, we have
\[ 
f_i(\boldsymbol{x})=(-1)^n \int_{\bar{\Omega}_{\boldsymbol{x}}} H_{\boldsymbol{\lambda}}(\boldsymbol{x},\boldsymbol{t})\frac{\partial^{\boldsymbol{\lambda}} f_i(\boldsymbol{t})}{\Delta \boldsymbol{t}^{\boldsymbol{\lambda}}}\Delta \boldsymbol{t}.  \\
 \]
In the proof of Theorem \ref{THEO4}, replacing $\Omega_{\boldsymbol{x}}$ by $\bar{\Omega}_{\boldsymbol{x}}$ we get \eqref{eq4'}. \qed
\end{proof}
\begin{remark}
Since  \eqref{eqRe1}, note that when $n= \boldsymbol{\lambda}=\alpha = 1, \beta = p-1,$ for $1<p\leq 2$,  $ m=2, G(x_1, x_2) =|x_1x_2|,$
$\omega_1=\psi^{\sigma}$, $\tau_1=\phi [\psi^{\sigma}]^{p/2}$, where $\phi, \psi \in \mathcal{W}([a,b]_{\mathbb T})$ and $\psi$ is decreasing  on $[a,b]_{\mathbb T}$, we see that  inequality \eqref{eq4'} improves \cite[Theorem 3.2]{YZ}.
\end{remark}

By applying Theorems \ref{THEO4} and \ref{THEO4'} on $\Omega_{\textbf{c}}$ and $\bar{\Omega}_{\textbf{c}}$, respectively, with $\textbf{c} \in \Omega^{\kappa^{\boldsymbol{\lambda-1}}}$ is such that $\Omega^{\kappa^{\boldsymbol{\lambda-1}}} =\Omega_{\textbf{c}}\cup \bar{\Omega}_{\textbf{c}}$, and summing the resulting inequalities, we have the following result.
\begin{theorem}\label{THEO4''}
Let $\omega_i, \tau_i \in \mathcal{W}(\Omega^{\kappa^{\boldsymbol{\lambda-1}}})$, and    $f_i : \Omega \to (-R, R)$ be such that
 \[\int_{\Omega^{\kappa^{\boldsymbol{\lambda-1}}}}\bigg|\frac{\partial^{\boldsymbol{\lambda}} f_i(\textbf{x}) }{\Delta \textbf{x}^{\boldsymbol{\lambda}}}\bigg|^{\alpha+\beta}\tau_i(\textbf{x})\Delta \textbf{x}<R \quad \text{for} \quad  i \in [1,m]_{\mathbb{N}}. \]

  If  $f_i \in  \mathcal{L}^{\alpha+\beta}_{\textbf{a}}(\Omega_{\textbf{c}}, \tau_i, \boldsymbol{\lambda}) \cap \mathcal{L}^{\alpha+\beta}_{{\boldsymbol \rho}^{\boldsymbol{\lambda-1}}(\textbf{b})}({\bar{\Omega}_{\textbf{c}}}, \tau_i, \boldsymbol{\lambda})  $ for all  $ i \in [1, m]_{\mathbb{N}}$,
then we have
\begin{equation}\label{eq4''}
\begin{split}
&\int_{\Omega^{\kappa^{\boldsymbol{\lambda-1}}}}\bigg(\sum^m_{i=1}{[D_iG(|f_1(\textbf{x})|, ..., |f_m(\textbf{x})|)]^\alpha\bigg|\frac{\partial^{\boldsymbol{\lambda}}f_i(\textbf{x}) }{\Delta \textbf{x}^{\boldsymbol\lambda}}\bigg|^\alpha\omega_i(\textbf{x})}\bigg) \Delta \textbf{x}\\
& \leq  K_{\Omega_{\textbf{c}}}\bigg[G\bigg(\int_{{\Omega}_{\textbf{c}}}\bigg|\frac{\partial^{\boldsymbol{\lambda}} f_1(\textbf{x}) }{\Delta \textbf{x}^{\boldsymbol{\lambda}}}\bigg|^{\alpha+\beta}\tau_1(\textbf{x})\Delta \textbf{x}, ..., \int_{{\Omega}_{\textbf{c}}}\bigg|\frac{\partial^{\boldsymbol{\lambda}}f_m(\textbf{x})  }{\Delta \textbf{x}^{\boldsymbol{\lambda}}}\bigg|^{\alpha+\beta}\tau_m(\textbf{x})\Delta \textbf{x}  \bigg)\bigg]^{\frac{\alpha}{\alpha+\beta}} \\
&\quad  +  K^*_{{\bar{\Omega}_{\textbf{c}}}} \bigg[G\bigg(\int_{\bar{\Omega}_{\textbf{c}}}\bigg|\frac{\partial^{\boldsymbol{\lambda}} f_1(\textbf{x}) }{\Delta \textbf{x}^{\boldsymbol{\lambda}}}\bigg|^{\alpha+\beta}\tau_1(\textbf{x})\Delta \textbf{x}, ..., \int_{\bar{\Omega}_{\textbf{c}}}\bigg|\frac{\partial^{\boldsymbol{\lambda}}f_m(\textbf{x})  }{\Delta \textbf{x}^{\boldsymbol{\lambda}}}\bigg|^{\alpha+\beta}\tau_m(\textbf{x})\Delta \textbf{x}  \bigg)\bigg]^{\frac{\alpha}{\alpha+\beta}}\\
\end{split}
\end{equation}
for all $\textbf{c} \in \Omega^{\kappa^{\boldsymbol{\lambda-1}}}$ is such that $\Omega^{\kappa^{\boldsymbol{\lambda-1}}} =\Omega_{\textbf{c}}\cup \bar{\Omega}_{\textbf{c}}$, where $K_{\Omega_{\textbf{c}}},  K^*_{{\bar{\Omega}_{\textbf{c}}}}$ are  defined as in \eqref{k2}  and  \eqref{k2'}, respectively.
\end{theorem} 

\begin{remark}
 If $n= \boldsymbol{\lambda}=\alpha = 1, \beta = p-1,$ for $1<p\leq 2$,  $ m=2, G(x_1, x_2) =|x_1x_2|,$
$\omega_1=\psi^{\sigma}$, $\tau_1=\phi [\psi^{\sigma}]^{p/2}$, where $\phi, \psi \in \mathcal{W}([a,b]_{\mathbb T})$ and $\psi$ is decreasing  on $[a,b]_{\mathbb T}$, we see that  inequality \eqref{eq4''} improves \cite[Theorem 3.2]{YZ}.
\end{remark}

Let us mention some important consequences of Theorems \ref{THEO4}, \ref{THEO4'}, and \ref{THEO4''}. First, let $m =1$; then we obtain the following corollary.
\begin{corollary}\label{coG}
Let $G\in \mathcal{G}^{1, 1}_R$, $\omega, \tau \in \mathcal{W}(\Omega)$, and 
  $f : \Omega \to (-R, R)$ be such that 
\[\int_{\Omega}\bigg|\frac{\partial^{\boldsymbol{\lambda}} f(\textbf{x}) }{\Delta \textbf{x}^{\boldsymbol{\lambda}}}\bigg|^{\alpha+\beta}\tau(\textbf{x})\Delta \textbf{x}<R.\]

If  $f \in  \mathcal{L}^{\alpha+\beta}_{\textbf{a}}(\Omega, \tau, \boldsymbol{\lambda})$ and
\[N_{\Omega}:=
 \bigg[\int_{\Omega}[G'(  \vartheta(\textbf{x}))]^{{\frac{\alpha(\alpha+\beta-1)}{\beta}}}\omega^{\frac{\alpha+\beta}{\beta}}(\textbf{x})\tau^{-\frac{\alpha}{\beta}}(\textbf{x})  \Delta \textbf{x}\bigg]^{\frac{\beta}{\alpha+\beta}}<\infty,
\]
where
\[\vartheta(\textbf{x}): =\int_{\Omega_\textbf{x}} \bigg(H_{\boldsymbol{\lambda}}(\textbf{x},\textbf{t})\bigg)^{\frac{\alpha+\beta}{\alpha+\beta-1}}(\tau(\textbf{t}))^{\frac{1}{1-\alpha-\beta}}\Delta \textbf{t}, \quad \textbf{x} \in \Omega,\]
then
 \begin{equation}\label{ptG1}
\int_{\Omega}[G'(|f(\textbf{x})|)]^{\alpha}\bigg|\frac{\partial^{\boldsymbol{\lambda}}f(\textbf{x}) }{\Delta \textbf{x}^{\boldsymbol{\lambda}}}\bigg|^\alpha\omega(\textbf{x}) \Delta \textbf{x}\leq N_{\Omega}\bigg[G\bigg(\int_{\Omega}\bigg|\frac{\partial^{\boldsymbol{\lambda}} f(\textbf{x}) }{\Delta \textbf{x}^{\boldsymbol{\lambda}}}\bigg|^{\alpha+\beta}\tau(\textbf{x})\Delta \textbf{x}  \bigg)\bigg]^{\frac{\alpha}{\alpha+\beta}}. \\
\end{equation}

Similarly, if  $f \in \mathcal{L}^{\alpha+\beta}_{{\boldsymbol\rho}^{\boldsymbol{\lambda -1} }(\textbf{b})}(\Omega^{\kappa^{\boldsymbol{\lambda-1}}}, \tau, \boldsymbol{\lambda})$ and

\[N^*_{\Omega^{\kappa^{\boldsymbol{\lambda-1}}}}:=
 \bigg[\int_{\Omega^{\kappa^{\boldsymbol{\lambda-1}}}}[G'(  \vartheta^*(\textbf{x}))]^{  {\frac{\alpha(\alpha+\beta-1)}{\beta}} }\omega^{\frac{\alpha+\beta}{\beta}}(\textbf{x})\tau^{-\frac{\alpha}{\beta}}(\textbf{x})  \Delta \textbf{x}\bigg]^{\frac{\beta}{\alpha+\beta}}<\infty,
\]
where
\[\vartheta^*(\textbf{x}): =\int_{\bar{\Omega}_\textbf{x}} |H_{\boldsymbol{\lambda}}(\textbf{x},\textbf{t})|^{\frac{\alpha+\beta}{\alpha+\beta-1}}(\tau(\textbf{t}))^{\frac{1}{1-\alpha-\beta}}\Delta \textbf{t}, \quad \textbf{x} \in \Omega,\]
then, 
\begin{equation}\label{ptG2}
\begin{split}
&\int_{\Omega^{\kappa^{\boldsymbol{\lambda-1}}}}[G'(|f(\textbf{x})|)]^\alpha\bigg|\frac{\partial^{\boldsymbol{\lambda}}f(\textbf{x}) }{\Delta \textbf{x}^{\boldsymbol{\lambda}}}\bigg|^\alpha\omega(\textbf{x}) \Delta \textbf{x}\\
&\leq N^*_{\Omega^{\kappa^{\boldsymbol{\lambda-1}}}}\bigg[G\bigg(\int_{\Omega^{\kappa^{\boldsymbol{\lambda-1}}}}\bigg|\frac{\partial^{\boldsymbol{\lambda}} f(\textbf{x}) }{\Delta \textbf{x}^{\boldsymbol{\lambda}}}\bigg|^{\alpha+\beta}\tau(\textbf{x})\Delta \textbf{x}  \bigg)\bigg]^{\frac{\alpha}{\alpha+\beta}}. 
\end{split}
\end{equation}

If  $f \in \mathcal{L}^{\alpha+\beta}_{\textbf{a}}(\Omega_{\textbf{c}}, \tau, \boldsymbol{\lambda})  \cap \mathcal{L}^{\alpha+\beta}_{{{\boldsymbol\rho}^{\boldsymbol{\lambda -1} }(\textbf{b})}}(\bar{\Omega}_{\textbf{c}}, \tau, \boldsymbol{\lambda})$  and  $N_{\Omega_{\textbf{c}}}$,  $N^*_{\bar{\Omega}_{\textbf{c}}}$ are finite for all $\textbf{c} \in \Omega^{\kappa^{\boldsymbol{\lambda-1}}}$ is such that $\Omega^{\kappa^{\boldsymbol{\lambda-1}}} =\Omega_{\textbf{c}}\cup \bar{\Omega}_{\textbf{c}}$, then
\begin{equation}\label{ptG3} 
\begin{split}
&\int_{\Omega^{\kappa^{\boldsymbol{\lambda-1}}}}[G'(|f(\textbf{x})|)]^\alpha\bigg|\frac{\partial^{\boldsymbol{\lambda}}f(\textbf{x}) }{\Delta \textbf{x}^{\boldsymbol{\lambda}}}\bigg|^\alpha\omega(\textbf{x}) \Delta \textbf{x} \\
& \leq   N_{\Omega_{\textbf{c}}}\bigg[G\bigg(\int_{\Omega_{\textbf{c}} }\bigg|\frac{\partial^{\boldsymbol{\lambda}} f(\textbf{x}) }{\Delta \textbf{x}^{\boldsymbol{\lambda}}}\bigg|^{\alpha+\beta}\tau(\textbf{x})\Delta \textbf{x}  \bigg)\bigg]^{\frac{\alpha}{\alpha+\beta}}\\
&\quad + N^*_{\bar{\Omega}_{\textbf{c}}}\bigg[G\bigg(\int_{\bar{\Omega}_{\textbf{c}}}\bigg|\frac{\partial^{\boldsymbol{\lambda}} f(\textbf{x}) }{\Delta \textbf{x}^{\boldsymbol{\lambda}}}\bigg|^{\alpha+\beta}\tau(\textbf{x})\Delta \textbf{x}  \bigg)\bigg]^{\frac{\alpha}{\alpha+\beta}}. \\
\end{split}
\end{equation}
\end{corollary}

\begin{remark}
Inequality \eqref{ptG1} is the same as inequality given in \cite[Theorem 2.9 ]{SAK1}, if we take
$n=1$  and $G(x) = |x|^{\gamma}$ for $ \gamma>1.$  If $\lambda_1 = 1$, then 
inequalities \eqref{ptG1} and \eqref{ptG2}  reduce to inequalities given  in \cite[Theorems 3.1 and 3.2]{SORA}, respectively.
\end{remark}

We examine Corollary \ref{coG} further in the case when $G(x) = |x|^{\frac{\alpha+\beta}{\alpha}}$. We obtain the following corollary.
\begin{corollary}\label{Co3.16}
Let   $\omega, \tau \in \mathcal{W}(\Omega),$ and
$f : \Omega \to (-R, R)$ be such that 
\[\int_{\Omega}\bigg|\frac{\partial^{\boldsymbol{\lambda}} f(\textbf{x}) }{\Delta \textbf{x}^{\boldsymbol{\lambda}}}\bigg|^{\alpha+\beta}\tau(\textbf{x})\Delta \textbf{x}<R.\]

If $f \in \mathcal{L}^{\alpha+\beta}_{\textbf{a}}(\Omega, \tau, \boldsymbol{\lambda})$  and
\[ 
L_{\Omega}:= \bigg[  \int_{\Omega}  \bigg( \int_{\Omega_{\textbf{x}}} H^{{\frac{\alpha+\beta}{\alpha+\beta-1}}}_{\boldsymbol{\lambda}}(\textbf{x},\textbf{t}) \tau^{\frac{1}{1-\alpha-\beta}}(\textbf{t})  \Delta \textbf{t}  \bigg)^{\alpha+\beta-1}  \omega^{\frac{\alpha+\beta}{\beta}}(\textbf{x})\tau^{-\frac{\alpha}{\beta}}(\textbf{x})    \Delta \textbf{x}  \bigg]^{\frac{\beta}{\alpha+\beta}} < \infty,  
\]
then 
\begin{equation}\label{ptmin1}
 \int_{\Omega}|f(\textbf{x})|^\beta\bigg|\frac{\partial^{\boldsymbol{\lambda}}f(\textbf{x}) }{\Delta \textbf{x}^{\boldsymbol{\lambda}}}\bigg|^\alpha\omega(\textbf{x})   \Delta \textbf{x} \leq   \bigg(\frac{\alpha}{\alpha+\beta}\bigg)^{\frac{\alpha}{\alpha+\beta}}L_{\Omega} \int_{\Omega} \bigg|\frac{\partial^{\boldsymbol{\lambda}}f(\textbf{x}) }{\Delta \textbf{x}^{\boldsymbol{\lambda}}}\bigg|^{\alpha+\beta}\tau(\textbf{x})   \Delta \textbf{x}.
\end{equation}
 
Similarly, if $f \in \mathcal{L}^{\alpha+\beta}_{{\boldsymbol\rho}^{\boldsymbol{\lambda -1} }(\textbf{b})}(\Omega^{\kappa^{\boldsymbol{\lambda-1}}}, \tau, \boldsymbol{\lambda})$ and 
\[ 
L^*_{\Omega^{\kappa^{\boldsymbol{\lambda-1}}}}:= \bigg[  \int_{\Omega^{\kappa^{\boldsymbol{\lambda-1}}} }  \bigg( \int_{\bar{\Omega}_{\textbf{x}}} |H_{\boldsymbol{\lambda}}(\textbf{x},\textbf{t})|^{{\frac{\alpha+\beta}{\alpha+\beta-1}}} \tau^{\frac{1}{1-\alpha-\beta}}(\textbf{t})  \Delta \textbf{t}  \bigg)^{\alpha+\beta-1}  \omega^{\frac{\alpha+\beta}{\beta}}(\textbf{x})\tau^{-\frac{\alpha}{\beta}}(\textbf{x})    \Delta \textbf{x}  \bigg]^{\frac{\beta}{\alpha+\beta}}
\]
is finite, then
\begin{equation}\label{ptmin2}
\begin{split}
 &\int_{\Omega^{\kappa^{\boldsymbol{\lambda-1}}}}|f(\textbf{x})|^\beta\bigg|\frac{\partial^{\boldsymbol{\lambda}}f(\textbf{x}) }{\Delta \textbf{x}^{\boldsymbol{\lambda}}}\bigg|^\alpha\omega(\textbf{x})   \Delta \textbf{x}\\
& \leq   \bigg(\frac{\alpha}{\alpha+\beta}\bigg)^{\frac{\alpha}{\alpha+\beta}}L^*_{\Omega^{\kappa^{\boldsymbol{\lambda-1}}}} \int_{\Omega^{\kappa^{\boldsymbol{\lambda-1}}}} \bigg|\frac{\partial^{\boldsymbol{\lambda}}f(\textbf{x}) }{\Delta \textbf{x}^{\boldsymbol{\lambda}}}\bigg|^{\alpha+\beta}\tau(\textbf{x})   \Delta \textbf{x}.
\end{split}
\end{equation}

If  $f \in \mathcal{L}^{\alpha+\beta}_{\textbf{a}}(\Omega_{\textbf{c}}, \tau, \boldsymbol{\lambda}) \cap \mathcal{L}^{\alpha+\beta}_{{\boldsymbol\rho}^{\boldsymbol{\lambda -1} }(\textbf{b})}({\bar{\Omega}_{\textbf{c}}}, \tau, \boldsymbol{\lambda})$, $L_{\Omega_{\textbf{c}}}$ and $L^*_{{\bar{\Omega}_{\textbf{c}}}} $ are finite for $\textbf{c} \in \Omega^{\kappa^{\boldsymbol{\lambda-1}}}$ is such that $\Omega^{\kappa^{\boldsymbol{\lambda-1}}} = \Omega_{\textbf{c}}\cup {\bar{\Omega}_{\textbf{c}}}$, then 
\begin{equation}\label{ptmin}
\begin{split}
& \int_{\Omega^{\kappa^{\boldsymbol{\lambda-1}}}}|f(\textbf{x})|^\beta\bigg|\frac{\partial^{\boldsymbol{\lambda}}f(\textbf{x}) }{\Delta \textbf{x}^{\boldsymbol{\lambda}}}\bigg|^\alpha\omega(\textbf{x})   \Delta \textbf{x}\\
& \leq   \bigg(\frac{\alpha}{\alpha+\beta}\bigg)^{\frac{\alpha}{\alpha+\beta}} \bigg[L_{\Omega_{\textbf{c}}} \int_{   \Omega_{\textbf{c}}   } \bigg|\frac{\partial^{\boldsymbol{\lambda}}f(\textbf{x}) }{\Delta \textbf{x}^{\boldsymbol{\lambda}}}\bigg|^{\alpha+\beta}\tau(\textbf{x})   \Delta \textbf{x}\\
&\quad +L^*_{{\bar{\Omega}_{\textbf{c}}}}  \int_{{\bar{\Omega}_{\textbf{c}}}     } \bigg|\frac{\partial^{\boldsymbol{\lambda}}f(\textbf{x}) }{\Delta \textbf{x}^{\boldsymbol{\lambda}}}\bigg|^{\alpha+\beta}\tau(\textbf{x})   \Delta \textbf{x}\bigg].\\
\end{split}
\end{equation}
 Furthermore, if we choose   $\textbf{c}\in \Omega^{\kappa^{\boldsymbol{\lambda-1}}}$ is such that $ |L_{\Omega_{\textbf{c}}}-L^*_{{\bar{\Omega}_{\textbf{c}}}}| \to \min$, then  \eqref{ptmin} reduces to 
\begin{equation}\label{ptMIN}
\begin{split}
& \int_{\Omega^{\kappa^{\boldsymbol{\lambda-1}}}}|f(\textbf{x})|^\beta\bigg|\frac{\partial^{\boldsymbol{\lambda}}f(\textbf{x}) }{\Delta \textbf{x}^{\boldsymbol{\lambda}}}\bigg|^\alpha\omega(\textbf{x})   \Delta \textbf{x}\\
& \leq   \bigg(\frac{\alpha}{\alpha+\beta}\bigg)^{\frac{\alpha}{\alpha+\beta}}\max \{L_{\Omega_{\textbf{c}}},  L^*_{{\bar{\Omega}_{\textbf{c}}}}\}\int_{  \Omega^{\kappa^{\boldsymbol{\lambda-1}}} } \bigg|\frac{\partial^{\boldsymbol{\lambda}}f(\textbf{x}) }{\Delta \textbf{x}^{\boldsymbol{\lambda}}}\bigg|^{\alpha+\beta}\tau(\textbf{x})   \Delta \textbf{x}.
\end{split}
\end{equation}

\end{corollary}

\begin{remark}
The results in Corollary \ref{Co3.16} when  $n=1,  \boldsymbol{\lambda} = 1, \Omega = [a, b]_{\mathbb{T}},  f(a) = 0$ or/and $f(b) = 0$ yield many known results. Namely, if we set $a =0, \alpha = \beta = 1, \omega= \tau \equiv 1$, then inequality \eqref{ptmin} becomes
\begin{equation}\label{time1}
\begin{split}
&\int_0^b |f(x)| | f^\Delta(x)| \Delta x\\
 &\leq \frac{\sqrt{2}}{2} \bigg [\bigg(\int_0^c x \Delta x \bigg)^{1/2}  \int_0^c | f^\Delta(x)|^2 \Delta x + \bigg(\int_c^b (b-x) \Delta x \bigg)^{1/2} \int_c ^b | f^\Delta(x)|^2 \Delta x \bigg] \\
\end{split} 
\end{equation}
for all $c \in [0, b]_{\mathbb{T}}.$ Since $x \leq  \frac{1}{2}(x^2)^\Delta$, then inequality  \eqref{time1} reduces to
\begin{equation}\label{time2}
\int_0^b |f(x)| | f^\Delta(x)| \Delta x \leq \frac{c}{2} \int_0^c | f^\Delta(x)|^2 \Delta x + \frac{b-c}{2}\int_c^b | f^\Delta(x)|^2 \Delta x
\end{equation} 
for all $c \in [0, b]_{\mathbb{T}}.$ If $\mathbb{T} =\mathbb{R},$ then we can choose $c =b/2$ and inequality \eqref{time2}  yields inequality \eqref{1.1}.

Next, inequalities \eqref{ptmin1} and \eqref{ptmin2} are the same as the ones given in \cite[Theorems 2.4 and 2.5]{SAK}, respectively, if we take $ n =1.$
\begin{remark}
When $\mathbb{T} =\mathbb{R}$,  from Corollary \ref{Co3.16} we obtain many known results:
\begin{enumerate}
\item If $n=2$ and $\boldsymbol{\lambda} = (1,1)$, then inequalities \eqref{ptmin1} and \eqref{ptmin2} imply the results of Pachpatte given in  \cite[Theorems 1 and 4]{pach1}, respectively.
\item  Let $n=2, \boldsymbol{\lambda} = (\lambda_1,\lambda_2), \alpha=\beta =1,$ and $\omega = \tau \equiv 1$; then inequality \eqref{ptmin1} reduces to   \cite[Theorem 2.1]{Zhao}.
\end{enumerate}
\end{remark}
\end{remark}

The following result  which can be proved in view of Corollary \ref{coG} and the well-known inequality of arithmetic and geometric means:
\begin{equation}\label{mean}
\prod^k_{j =1}a_j \leq \bigg( \frac{1}{k}\sum^k_{j=1}a_j\bigg)^k, \quad a_j \geq 0 \quad \text{for}\quad j \in [1, k]_{\mathbb{N}}.
\end{equation}

\begin{corollary}
Let $k \in \mathbb{N}$, $\alpha_j ,\beta_j >0$ be such that $\alpha_j + \beta_j >1/k$,  $G_j\in \mathcal{G}^{1,1}_R$, and  $\omega_j$, $ \tau_j$  be in  $\mathcal{W}(\Omega)$, 
 and
$f_j : \Omega \to (-R, R)$ be such that 
$\int_{\Omega}|\frac{\partial^{\boldsymbol{\lambda}} f_j(\textbf{x}) }{\Delta \textbf{x}^{\boldsymbol{\lambda}}}|^{k(\alpha+\beta)}\tau_j(\textbf{x})\Delta \textbf{x}<R$ for $j \in [1, k]_{\mathbb{N}}.$

 If  $f_j \in \mathcal{L}^{k(\alpha_j+\beta_j)}_{\textbf{a}}(\Omega, \tau_j, \boldsymbol{\lambda})$ for all  $j \in [1, k]_{\mathbb{N}}$, and 
\[
{T_{\Omega}}_j:= \bigg[\int_{\Omega}[G'_j( \eta_j(\textbf{x}))]^{\frac{\alpha_j(k\alpha_j+k\beta_j-1)}{\beta_j}}\omega^{\frac{k(\alpha_k+\beta_i)}{\beta_j}}_j(\textbf{x})\tau^{-\frac{\alpha_j}{\beta_j}}_j(\textbf{x})  \Delta \textbf{x}\bigg]^{\frac{\beta_j}{\alpha_j+\beta_j}}<\infty,\]
where
\[ \eta_j (\textbf{x})  := \int_{\Omega_\textbf{x}} \bigg(H_{\boldsymbol{\lambda}}(\textbf{x},\textbf{t})\bigg)^{\frac{k(\alpha_j+\beta_j)}{k(\alpha_j+\beta_j)-1}}(\tau_j(\textbf{t}))^{\frac{1}{1-k(\alpha_j+\beta_i)}}\Delta \textbf{t}, \quad \textbf{x} \in \Omega, \]
then 
\begin{equation}\label{phtich1}
\begin{split}
&\int_{\Omega}\prod^{k}_{j=1}{[G'_j(|f_j(\textbf{x})|)]^{\alpha_j}\bigg|\frac{\partial^{ \boldsymbol{\lambda}   }f_j(\textbf{x}) }{\Delta \textbf{x}^{  \boldsymbol{\lambda}  }}\bigg|^{\alpha_j}\omega_j(\textbf{x})} \Delta \textbf{x} \\
&\leq \frac{1}{k}\sum^k_{j=1}{T_{\Omega j} \bigg[G_j\bigg(\int_{\Omega}\bigg|\frac{\partial^{ \boldsymbol{\lambda}  } f_j(\textbf{x}) }{\Delta \textbf{x}^{  \boldsymbol{\lambda}  }}\bigg|^{k(\alpha_j+\beta_j)}\tau_j(\textbf{x})\Delta \textbf{x}  \bigg)\bigg]^{\frac{\alpha_j}{\alpha_j+\beta_j}}}. \\
\end{split}
\end{equation}

If $f_j \in \mathcal{L}^{k(\alpha+\beta)}_{{\boldsymbol\rho}^{\boldsymbol{\lambda -1} }(\textbf{b})}(\Omega^{\kappa^{\boldsymbol{\lambda-1}}}, \tau_j, \boldsymbol{\lambda})$ for all $j \in [1, k]_{\mathbb{N}}$,  and 
\[
T^*_{\Omega j}:= \bigg[\int_{\Omega}[G'_j(\eta^*_j(\textbf{x}))]^{\frac{\alpha_j(k\alpha_j+k\beta_j-1)}{\beta_j}}\omega^{\frac{k(\alpha_k+\beta_i)}{\beta_j}}_j(\textbf{x})\tau^{-\frac{\alpha_j}{\beta_j}}_j(\textbf{x})  \Delta \textbf{x}\bigg]^{\frac{\beta_j}{\alpha_j+\beta_j}}<\infty, \]
where 
\[ \eta^*_j(\textbf{x})  := \int_{\bar{\Omega}_\textbf{x}} |H_{\boldsymbol{\lambda}}(\textbf{x},\textbf{t})|^{\frac{k(\alpha_j+\beta_j)}{k(\alpha_j+\beta_j)-1}}(\tau_j(\textbf{t}))^{\frac{1}{1-k(\alpha_j+\beta_i)}}\Delta \textbf{t}, \quad \textbf{x} \in \Omega^{\kappa^{\boldsymbol{\lambda-1}}}, \]
then
\begin{equation}\label{phtich2}
\begin{split}
&\int_{\Omega^{\kappa^{\boldsymbol{\lambda-1}}}}\prod^{k}_{j=1}{[G'_j(|f_j(\textbf{x})|)]^{\alpha_j}\bigg|\frac{\partial^{ \boldsymbol{\lambda}   }f_j(\textbf{x}) }{\Delta \textbf{x}^{  \boldsymbol{\lambda}  }}\bigg|^{\alpha_j}\omega_j(\textbf{x})} \Delta \textbf{x} \\
&\leq \frac{1}{k}\sum^k_{j=1}{T^*_{\Omega j} \bigg[G_j\bigg(\int_{\Omega^{\kappa^{\boldsymbol{\lambda-1}}}}\bigg|\frac{\partial^{\boldsymbol{\lambda}   } f_j(\textbf{x}) }{\Delta \textbf{x}^{  \boldsymbol{\lambda}  }}\bigg|^{k(\alpha_j+\beta_j)}\tau_j(\textbf{x})\Delta \textbf{x}  \bigg)\bigg]^{\frac{\alpha_j}{\alpha_j+\beta_j}}}. \\
\end{split}
\end{equation}
\end{corollary}

\begin{proof}
By using  inequality \eqref{mean}, we obtain
\[
\begin{split}
&\prod^{k}_{j=1}{[G'_j(|f_j(\boldsymbol{x})|)]^{\alpha_j}\bigg|\frac{\partial^{  \boldsymbol{\lambda}   }f_j(\boldsymbol{x}) }{\Delta \boldsymbol{x}^{  \boldsymbol{\lambda}    }}\bigg|^{\alpha_j}\omega_j(\boldsymbol{x})}   \\
&=\bigg(\prod^{k}_{j=1}{[G'_j(|f_j(\boldsymbol{x})|)]^{\alpha_j k}\bigg|\frac{\partial^{    \boldsymbol{\lambda}  }f_j(\boldsymbol{x}) }{\Delta \boldsymbol{x}^{   \boldsymbol{\lambda}   }}\bigg|^{\alpha_j k}\omega^k_j(\boldsymbol{x})}\bigg)^{\frac{1}{k}}   \\
& \leq \frac{1}{k}\sum^k_{j=1}{{[G'_j(|f_j(\boldsymbol{x})|)]^{\alpha_j k}\bigg|\frac{\partial^{     \boldsymbol{\lambda}  }f_j(\boldsymbol{x}) }{\Delta \boldsymbol{x}^{   \boldsymbol{\lambda}   }}\bigg|^{\alpha_j k}\omega^k_j(\boldsymbol{x})}  }. \\
\end{split}
\]
Applying inequalities \eqref{ptG1} and \eqref{ptG2} we arrive at \eqref{phtich1} and  \eqref{phtich2}. \qed
\end{proof}

The following theorem generalizes Rozanova's  inequality \cite[Theorem 2.5]{AND} to functions of several variables on time scales.
\begin{theorem}\label{THEO5}
Let $F \in \mathcal{H}^m_{\infty}$, $\phi_i$ be convex, non-negative, and increasing on $[0, \infty)$, and $\varphi_i: \Omega \to \mathbb{R}$ be such that $\frac{\partial^{\boldsymbol{1}} \varphi_i(\textbf{x})}{\Delta \textbf{x}^{\boldsymbol{1}}}$ is non-negative with $\frac{\partial^{k_j} \varphi_i(\textbf{x})}{\Delta_jx^{k_j}_j}|_{x_j = a_j} = 0,$ where $ k_j \in [0, \lambda_j-1]_{\mathbb{N}}$,  $  j \in [1, n]_{\mathbb{N}},$ and  $i \in [1, m]_{\mathbb{N}}$.  For any $  f_i \in \text{C}^{n \boldsymbol{\lambda}}_{\text{rd}}(\Omega) $ is such that  $\frac{\partial^{k_j} f_i(\textbf{x})}{\Delta_jx^{k_j}_j}|_{x_j = a_j} = 0$   for all $ k_j \in [0, \lambda_j-1]_{\mathbb{N}}$,  $  j \in [1, n]_{\mathbb{N}},$ and $i \in [1, m]_{\mathbb{N}}$, we have
\begin{equation}\label{ptTHEO5}
\begin{split}
&\int_{\Omega}\bigg[\sum^m_{i=1}D_iF\bigg(\varphi_1(\textbf{x})\phi_1\bigg(\frac{|f_1(\textbf{x})|}{\varphi_1(\textbf{x})}\bigg), ..., \varphi_m(\textbf{x})\phi_m\bigg(\frac{|f_m(\textbf{x})|}{\varphi_m(\textbf{x})}\bigg)\bigg) \\
&\quad \times \frac{\partial^{\boldsymbol{1}} \varphi_i(\textbf{x}) }{\Delta \textbf{x}^{\boldsymbol{1}} } \phi_i\bigg(H_{  \boldsymbol{\lambda}  }(\textbf{b},\textbf{a})\frac{|\frac{\partial^{  \boldsymbol{\lambda} } f_i(\textbf{x})  }{\Delta \textbf{x}^{  \boldsymbol{\lambda}    }}|}{\frac{\partial^{\boldsymbol{1}} \varphi_i(\textbf{x})}{\Delta \textbf{x}^{\boldsymbol{1}} }}   \bigg)\bigg]\Delta \textbf{x} \\
&\leq  F\bigg( \int_{\Omega}   \frac{\partial^{\boldsymbol{1}} \varphi_1(\textbf{x})}{\Delta \textbf{x}^{\boldsymbol{1}}}\phi_1\bigg( H_{  \boldsymbol{\lambda}    }(\textbf{b},\textbf{a})\frac{ |\frac{\partial^{   \boldsymbol{\lambda}  } f_1(\textbf{x}) }{\Delta \textbf{x}^{  \boldsymbol{\lambda}   }}|}{\frac{\partial^{\boldsymbol{1}} \varphi_1(\textbf{x}) }{\Delta \textbf{x}^{\boldsymbol{1}}}}\bigg) \Delta \textbf{x}, ..., \\
&\quad \int_{\Omega}  \frac{\partial^{\boldsymbol{1}} \varphi_m(\textbf{x}) }{\Delta \textbf{x}^{\boldsymbol{1}}}\phi_m\bigg(H_{\boldsymbol{\lambda}}(\textbf{b},\textbf{a})\frac{| \frac{\partial^{ \boldsymbol{\lambda}    }f_m(\textbf{x})}{\Delta \textbf{x}^{   \boldsymbol{\lambda}   }}|}{\frac{\partial^{\boldsymbol{1}} \varphi_m(\textbf{x})}{\Delta \textbf{x}^{\boldsymbol{1}}}} \bigg) \Delta \textbf{x} \bigg).  \\
\end{split}
\end{equation}

\end{theorem}

\begin{proof}
As in  the proof of Theorem \ref{THEO3}, for all $\boldsymbol{x} \in \Omega$, and all $i \in [1, m]_{\mathbb{N}},$ we have
\[
|f_i(\boldsymbol{x})| \leq \int_{\Omega_{\boldsymbol{x}}} H_{    \boldsymbol{\lambda}  }(\boldsymbol{x},\boldsymbol{t})\bigg|\frac{\partial^{    \boldsymbol{\lambda}  } f_i(\boldsymbol{t})}{\Delta \boldsymbol{t}^{   \boldsymbol{\lambda}   }}\bigg| \Delta \boldsymbol{t} := y_i(\boldsymbol{x}).
\]
We see that
\[
y_i(\boldsymbol{x}) =\int_{\Omega_{\boldsymbol{x}}} H_{    \boldsymbol{\lambda}  }(\boldsymbol{x},\boldsymbol{t})\frac{\partial^{    \boldsymbol{\lambda}  } y_i(\boldsymbol{t})}{\Delta \boldsymbol{t}^{   \boldsymbol{\lambda}   }} \Delta \boldsymbol{t} \leq   H_{    \boldsymbol{\lambda}  }(\boldsymbol{b},\boldsymbol{a})  \int_{\Omega_{\boldsymbol{x}}}\frac{\partial^{    \boldsymbol{\lambda}  } y_i(\boldsymbol{t})}{\Delta \boldsymbol{t}^{   \boldsymbol{\lambda}   }}\Delta \boldsymbol{t}, \quad \boldsymbol{x} \in \Omega.
\]
By  using the fact that  $\phi_i, i \in [1, m]_{\mathbb{N}},$ are non-negative, and increasing on $[0, \infty)$, we get
\[
\begin{split}
\phi_i\bigg(\frac{|f_i(\boldsymbol{x})|}{\varphi_i(\boldsymbol{x})}\bigg) &\leq \phi_i\bigg(\frac{y_i(\boldsymbol{x})}{\varphi_i(\boldsymbol{x})}\bigg)\\
&\leq  \phi_i\bigg(H_{   \boldsymbol{\lambda}    }(\boldsymbol{b},\boldsymbol{a})\frac{\int_{\Omega_{\boldsymbol{x}}}\frac{\partial^{     \boldsymbol{\lambda}   } y_i(\boldsymbol{x}) }{\Delta \boldsymbol{t}^{    \boldsymbol{\lambda}   }}\Delta \boldsymbol{t}}{\int_{\Omega_{\boldsymbol{x}}}\frac{\partial^{\boldsymbol{1}}\varphi_i(\boldsymbol{t})}{\Delta \boldsymbol{t}^{\boldsymbol{1}}}\Delta \boldsymbol{t}}\bigg)  \\
&=  \phi_i\bigg(H_{     \boldsymbol{\lambda}   }(\boldsymbol{b},\boldsymbol{a}) \int_{\Omega_{\boldsymbol{x}}}\frac{\frac{\partial^{{\boldsymbol{1}}} \varphi_i(\boldsymbol{t}) }{\Delta \boldsymbol{t}^{\boldsymbol{1}}}\frac{\partial^{    \boldsymbol{\lambda}    }y_i(\boldsymbol{t})}{\Delta \boldsymbol{t}^{     \boldsymbol{\lambda}   }} }{\frac{\partial^{{\boldsymbol{1}}}\varphi_i(\boldsymbol{t}) }{\Delta \boldsymbol{t}^{\boldsymbol{1}}}} \Delta \boldsymbol{t} \bigg),   \\
\end{split}
\]
which, in view of Jensen's inequality \cite{Jen}, we obtain
\begin{equation}\label{ptphi}
 \phi_i\bigg(\frac{|f_i(\boldsymbol{x})|}{\varphi_i(\boldsymbol{x})}\bigg) \leq \frac{1}{\varphi_i(\boldsymbol{x})}\int_{\Omega_{\boldsymbol{x}}}\frac{\partial^{\boldsymbol{1}} \varphi_i(\boldsymbol{t})}{\Delta \boldsymbol{t}^{\boldsymbol{1}}}\phi_i\bigg( H_{    \boldsymbol{\lambda}   }(\boldsymbol{b},\boldsymbol{a})\frac{ \frac{\partial^{     \boldsymbol{\lambda}    } y_i(\boldsymbol{t})}{\Delta \boldsymbol{t}^{     \boldsymbol{\lambda}   }}}{\frac{\partial^{\boldsymbol{1}}\varphi_i(\boldsymbol{t}) }{\Delta \boldsymbol{t}^{\boldsymbol{1} }}}\bigg) \Delta \boldsymbol{t}
\end{equation}
for  $\boldsymbol{x} \in \Omega$ and $i \in [1, m]_{\mathbb{N}}.$
 
Setting
\[
W_i(\boldsymbol{t}) :=\frac{\partial^{\boldsymbol{1}} \varphi_i(\boldsymbol{t}) }{\Delta \boldsymbol{t}^{\boldsymbol{1} }}\phi_i\bigg(H_{      \boldsymbol{\lambda} }(\boldsymbol{b},\boldsymbol{a})\frac{ \frac{\partial^{     \boldsymbol{\lambda}   } y_i(\boldsymbol{t})  }{\Delta \boldsymbol{t}^{     \boldsymbol{\lambda}   }}}{\frac{\partial^{\boldsymbol{1}}\varphi_i(\boldsymbol{t}) }{\Delta \boldsymbol{t}^{\boldsymbol{1}}}}\bigg), \quad \boldsymbol{t} \in \Omega_{\boldsymbol{x}},  \quad i \in [1, m]_{\mathbb{N}}.
\]
Thus, \eqref{ptphi} implies that
\begin{equation}\label{eqW}
\varphi_i(\boldsymbol{x})\phi_i\bigg(\frac{|f_i(\boldsymbol{x})|}{\varphi_i(\boldsymbol{x})}\bigg) \leq \int_{\Omega_{\boldsymbol{x}}}W_i(\boldsymbol{t})\Delta \boldsymbol{t},  \quad \boldsymbol{x} \in \Omega, \quad i \in [1, m]_{\mathbb{N}}.
\end{equation}

Since $D_iF$ for $i \in [1, m]_{\mathbb{N}}$, are non-negative,  increasing in each variable on $(0,  \infty)$ and \eqref{eqW}, we get
\[
\begin{split}
&\int_{\Omega}\bigg[\sum^m_{i=1}D_iF\bigg(\varphi_1(\boldsymbol{x})\phi_1\bigg(\frac{|f_1(\boldsymbol{x})|}{\varphi_1(\boldsymbol{x})}\bigg), ..., \varphi_m(\boldsymbol{x})\phi_m\bigg(\frac{|f_m(\boldsymbol{x})|}{\varphi_m(\boldsymbol{x})}\bigg)\bigg)\\
&\quad \times \frac{\partial^{\boldsymbol{1}}\varphi_i(\boldsymbol{x}) }{\Delta \boldsymbol{x}^{\boldsymbol{1}}} \phi_i\bigg(H_{      \boldsymbol{\lambda} }(\boldsymbol{b},\boldsymbol{a}) \frac{|\frac{\partial^{  \boldsymbol{\lambda}  } f_i(\boldsymbol{x})}{\Delta \boldsymbol{x}^{     \boldsymbol{\lambda}   }}|}{\frac{\partial^{\boldsymbol{1}}\varphi_i(\boldsymbol{x}) }{\Delta \boldsymbol{x}^{\boldsymbol{1}}}}   \bigg)\bigg]\Delta \boldsymbol{x} \\
&\leq \int_{\Omega}\bigg[\sum^m_{i=1}D_iF\bigg(\int_{\Omega_{\boldsymbol{x}}}W_1(\boldsymbol{t})\Delta \boldsymbol{t}, ...,\int_{\Omega_{\boldsymbol{x}}}W_m(\boldsymbol{t})\Delta \boldsymbol{t}   \bigg)W_i(\boldsymbol{x})\bigg]\Delta \boldsymbol{x}. \\
\end{split}
\]
Using \eqref{ptTHEO1}, we obtain \eqref{ptTHEO5}. \qed
\end{proof}

\begin{remark}
\cite[Theorem 2.5]{AND} is a special case of Theorem \ref{THEO5} when $\mathbb{T} = \mathbb{R}$.
\end{remark}

The following theorem is an interesting generalization of Theorem \ref{THEO4}.
\begin{theorem}\label{THEO6}
Let $ G \in  \mathcal{G}^{1, m}_R$ and $\omega_i, \tau_i  \in \mathcal{W}(\Omega)$ for $\in [1, m]_{\mathbb{N}} $ be  such that
\[ P:=\bigg[ \int_{\Omega}{\bigg( \sum^m_{i=1}{D_iG(\nu_1(\textbf{x}), ..., \nu_m(\textbf{x}))\omega^{\frac{\alpha+\beta}{\beta}}_i(\textbf{x})\tau^{-\frac{\alpha}{\beta}}_i(\textbf{x})} \bigg)\Delta \textbf{x}}    \bigg]^{\frac{\beta}{\alpha+\beta}}<\infty, \]
where 
\[ \nu_i(\textbf{x}) = \bigg(\int_{\Omega_\textbf{x}}(\tau_i(\textbf{t}))^{\frac{1}{1-\alpha-\beta}}\Delta \textbf{t} \bigg)^{\frac{\alpha(\alpha+\beta-1)}{\beta}}, \quad \textbf{x} \in \Omega.  \]
Further, for  $i \in [1, m]_{\mathbb{N}}$,  let $\phi_i$ be convex, non-negative, and increasing on $[0, \infty)$, and  $\varphi_i: \Omega \to \mathbb{R}$ be such that $\frac{\partial^{\boldsymbol{1}} \varphi_i (\textbf{x})}{\Delta \textbf{x}^{\boldsymbol{1}}}$ is non-negative with $\frac{\partial^{k_j} \varphi_i(\textbf{x})}{\Delta_jx^{k_j}_j}|_{x_j = a_j} = 0,$ where  $ k_j \in [0, \lambda_j-1]_{\mathbb{N}}$,  $  j \in [1, n]_{\mathbb{N}}.$  If $  f_i \in \text{C}^{n \boldsymbol{\lambda}}_{\text{rd}}(\Omega) $ is such that  $\frac{\partial^{k_j} f_i(\textbf{x})}{\Delta_jx^{k_j}_j}|_{x_j = a_j} = 0$ for all $ k_j \in [0, \lambda_j-1]_{\mathbb{N}}$,  $  j \in [1, n]_{\mathbb{N}},$ and $i \in [1, m]_{\mathbb{N}}$,    and
\[  \int_{\Omega} \bigg|  \frac{\partial^{\boldsymbol{1}} \varphi_i(\textbf{x})}{\Delta \textbf{x}^{\boldsymbol{1}}}\phi_i\bigg( H_{ \boldsymbol{\lambda}    }(\textbf{b},\textbf{a})\frac{|\frac{\partial^{  \boldsymbol{\lambda}    } f_i(\textbf{x}) }{\Delta \textbf{x}^{\boldsymbol{\lambda}     }}|}{\frac{\partial^{\boldsymbol{1}} \varphi_i(\textbf{x}) }{\Delta \textbf{x}^{\boldsymbol{1}} }}\bigg)\bigg|^{\alpha+\beta}\tau_i(\textbf{x}) \Delta \textbf{x}<R   \]
 for all $i \in [1, m]_{\mathbb{N}}$, then we have
\begin{equation}\label{pt37}
\begin{split}
&\int_{\Omega}\bigg[\sum^m_{i=1}D_iG\bigg(\varphi^\alpha_1(\textbf{x})\phi^\alpha_1\bigg(\frac{|f_1(\textbf{x})|}{\varphi_1(\textbf{x})}\bigg), ..., \varphi^\alpha_m(\textbf{x})\phi^\alpha_m\bigg(\frac{|f_m(\textbf{x})|}{\varphi_m(\textbf{x})}\bigg)\bigg)\\
&\quad \times   \bigg|\frac{\partial^{\boldsymbol{1}} \varphi_i(\textbf{x}) }{\Delta \textbf{x}^{\boldsymbol{1}} } \phi_i\bigg(H_{ \boldsymbol{\lambda}      }(\textbf{b},\textbf{a})\frac{|\frac{\partial^{  \boldsymbol{\lambda}   }f_i(\textbf{x})  }{\Delta \textbf{x}^{  \boldsymbol{\lambda}    }}|}{\frac{\partial^{\boldsymbol{1}} \varphi_i(\textbf{x})}{\Delta \textbf{x}^{\boldsymbol{1}} }}   \bigg)\bigg|^\alpha\omega_i(\textbf{x})\bigg]\Delta \textbf{x} \\
&\leq P\bigg[G\bigg( \int_{\Omega} \bigg|  \frac{\partial^{\boldsymbol{1}} \varphi_1(\textbf{x})}{\Delta \textbf{x}^{\boldsymbol{1}}}\phi_1\bigg( H_{ \boldsymbol{\lambda}    }(\textbf{b},\textbf{a})\frac{ |\frac{\partial^{  \boldsymbol{\lambda}    } f_1(\textbf{x}) }{\Delta \textbf{x}^{\boldsymbol{\lambda}     }}|}{\frac{\partial^{\boldsymbol{1}} \varphi_1(\textbf{x}) }{\Delta \textbf{x}^{\boldsymbol{1}} }}\bigg)\bigg|^{\alpha+\beta}\tau_1(\textbf{x}) \Delta \textbf{x}, ...,\\
 &\quad \int_{\Omega} \bigg|  \frac{\partial^{\boldsymbol{1}} \varphi_m(\textbf{x}) }{\Delta \textbf{x}^{\boldsymbol{1}}}\phi_m\bigg(H_{   \boldsymbol{\lambda}  }(\textbf{b},\textbf{a})\frac{| \frac{\partial^{   \boldsymbol{\lambda}  }f_m(\textbf{x})}{\Delta  \textbf{x}^{  \boldsymbol{\lambda}   }} |}{\frac{\partial^{\boldsymbol{1}} \varphi_m(\textbf{x})}{\Delta \textbf{x}^{\boldsymbol{1}}}} \bigg)\bigg|^{\alpha+\beta}\tau_m(\textbf{x}) \Delta \textbf{x} \bigg)\bigg]^{\frac{\alpha}{\alpha+\beta}}. \\
\end{split}
\end{equation}
\end{theorem}

\begin{proof}
As in the proof of Theorem \ref{THEO5}, using \eqref{eqW} and  H\"{o}lder's inequality with  indices $(\alpha+\beta)/(\alpha+\beta-1)$ and $(\alpha+\beta)$, we obtain
\[
\begin{split}
\varphi_i(\boldsymbol{x})\phi_i\bigg(\frac{|f_i(\boldsymbol{x})|}{\varphi_i(\boldsymbol{x})}\bigg) &\leq \int_{\Omega_{\boldsymbol{x}}}W_i(\boldsymbol{t})\Delta \boldsymbol{t}\\
& \leq \bigg(\int_{\Omega_{\boldsymbol{x}}}(\tau_i(\boldsymbol{t}))^{\frac{1}{1-\alpha-\beta}}\Delta \boldsymbol{t} \bigg)^{\frac{\alpha+\beta-1}{\alpha+\beta}}  \bigg( \int_{\Omega_{\boldsymbol{x}}}(W_i(\boldsymbol{t}))^{\alpha+\beta} \tau_i(\boldsymbol{t})\Delta \boldsymbol{t}\bigg)^{\frac{1}{\alpha+\beta}}\\
\end{split}
\]
for  $\boldsymbol{x} \in \Omega$ and $i \in [1, m]_{\mathbb{N}}.$ Hence
\begin{align}
\notag&\varphi^\alpha_i(\boldsymbol{x})\phi^\alpha_i\bigg(\frac{|f_i(\boldsymbol{x})|}{\varphi_i(\boldsymbol{x})}\bigg)\\
 \notag&\leq  \bigg(\int_{\Omega_{\boldsymbol{x}}}(\tau_i(\boldsymbol{t}))^{\frac{1}{1-\alpha-\beta}}\Delta \boldsymbol{t} \bigg)^{\frac{\alpha(\alpha+\beta-1)}{\alpha+\beta}}    \bigg( \int_{\Omega_{\boldsymbol{x}}}(W_i(\boldsymbol{t}))^{\alpha+\beta}\tau_i(\boldsymbol{t})\Delta \boldsymbol{t}\bigg)^{\frac{\alpha}{\alpha+\beta}}\\
&  = (\nu_i(\boldsymbol{x}))^{\frac{\beta}{\alpha+\beta}} \bigg( \int_{\Omega_{\boldsymbol{x}}}(W_i(\boldsymbol{t}))^{\alpha+\beta}\tau_i(\boldsymbol{t})\Delta \boldsymbol{t}\bigg)^{\frac{\alpha}{\alpha+\beta}}, \quad \boldsymbol{x} \in \Omega, \quad i \in [1, m]_{\mathbb{N}}. 
\label{ww}
\end{align}

Thus, by $ G \in  \mathcal{G}^{1, m}_R$ and \eqref{ww}, we get 
\[
\begin{split}
&\int_{\Omega}\bigg[\sum^m_{i=1}D_iG\bigg(\varphi^\alpha_1(\boldsymbol{x})\phi^\alpha_1\bigg(\frac{|f_1(\boldsymbol{x})|}{\varphi_1(\boldsymbol{x})}\bigg), ..., \varphi^\alpha_m(\boldsymbol{x})\phi^\alpha_m\bigg(\frac{|f_m(\boldsymbol{x})|}{\varphi_m(\boldsymbol{x})}\bigg)\bigg)\\
&\quad \times \bigg| \frac{\partial^{\boldsymbol{1}}\varphi_i(\boldsymbol{x}) }{\Delta \boldsymbol{x}^{\boldsymbol{1} }}   \phi_i\bigg(H_{ \boldsymbol{\lambda}   }(\boldsymbol{b},\boldsymbol{a}) \frac{|\frac{\partial^{    \boldsymbol{\lambda}    } f_i(\boldsymbol{x})}{\Delta \boldsymbol{x}^{   \boldsymbol{\lambda}   }}|}{\frac{\partial^{\boldsymbol{1}}\varphi_i(\boldsymbol{x}) }{\Delta \boldsymbol{x}^{\boldsymbol{1}}}}  \bigg)\bigg|^\alpha \omega_i(\boldsymbol{x})  \bigg]\Delta \boldsymbol{x}, \\
&\leq \int_{\Omega}\bigg[\sum^m_{i=1}(D_iG(\nu_1(\boldsymbol{x}), ..., \nu_m(\boldsymbol{x})))^{\frac{\beta}{\alpha+\beta}}\bigg(D_iG\bigg(\int_{\Omega_{\boldsymbol{x}}}(W_1(\boldsymbol{t}))^{\alpha+\beta}\tau_1(\boldsymbol{t})\Delta \boldsymbol{t}, ...,\\
&\quad  \int_{\Omega_{\boldsymbol{x}}}(W_m(\boldsymbol{t}))^{\alpha+\beta}\tau_m(\boldsymbol{t})\Delta \boldsymbol{t} \bigg)\bigg)^{\frac{\alpha}{\alpha+\beta}} \bigg| \frac{\partial^{\boldsymbol{1}}\varphi_i(\boldsymbol{x}) }{\Delta \boldsymbol{x}^{\boldsymbol{1}} }   \phi_i\bigg(H_{\boldsymbol{\lambda}   }(\boldsymbol{b},\boldsymbol{a}) \frac{|\frac{\partial^{   \boldsymbol{\lambda}  } f_i(\boldsymbol{x})}{\Delta \boldsymbol{x}^{   \boldsymbol{\lambda}   }}|}{\frac{\partial^{\boldsymbol{1}}\varphi_i(\boldsymbol{x}) }{\Delta \boldsymbol{x}^{\boldsymbol{1}}}}  \bigg)\bigg|^\alpha \omega_i(\boldsymbol{x})  \bigg]\Delta \boldsymbol{x}, \\
\end{split}
\]
which, applying H\"{o}lder's inequality with  indices $(\alpha+\beta)/\beta$ and $(\alpha+\beta)/\alpha$, yields
\[
\begin{split}
& \leq \bigg[  \int_{\Omega}{\bigg( \sum^m_{i=1}{D_iG(\nu_1(\boldsymbol{x}), ..., \nu_m(\boldsymbol{x}))\omega^{\frac{\alpha+\beta}{\beta}}_i(\boldsymbol{x})\tau^{-\frac{\alpha}{\beta}}_i(\boldsymbol{x})} \bigg)\Delta \boldsymbol{x}}    \bigg]^{\frac{\beta}{\alpha+\beta}}  \\
 &\quad \times \bigg[   \int_{\Omega}\bigg(\sum^m_{i=1}D_iG\bigg(\int_{\Omega_{\boldsymbol{x}}}(W_1(\boldsymbol{t}))^{\alpha+\beta}\tau_1(\boldsymbol{t})\Delta \boldsymbol{t}, ...,\\
&\qquad \int_{\Omega_{\boldsymbol{x}}}(W_m(\boldsymbol{t}))^{\alpha+\beta}\tau_m(\boldsymbol{t})\Delta \boldsymbol{t}  \bigg)(W_i(\boldsymbol{x}))^{\alpha+\beta}\tau_i(\boldsymbol{x})\bigg)\Delta \boldsymbol{x} \bigg]^{\frac{\alpha}{\alpha+\beta}},    \\
\end{split}
\]
which completes the proof  by applying \eqref{ptTHEO1}. \qed
\end{proof}

\section{Applications}
 In this section, we use Opial-type inequalities to establish  Lyapunov-type inequalities for the following half-linear dynamic equation
\begin{equation}\label{Lya}
\frac{\partial^{\boldsymbol{1}}}{\Delta \boldsymbol{x}^{\boldsymbol{1}}} \bigg(r(\boldsymbol{x})\dfrac{\partial^{\boldsymbol{1}} y(\boldsymbol{x})}{\Delta\boldsymbol{x}^{\boldsymbol{1}}} \bigg)  + s(\boldsymbol{x})y^{\boldsymbol\sigma}(\boldsymbol{x}) =0 \quad \text{on } \quad D = [a_1,b_1]_{\mathbb{T}_1}\times[a_2,b_2]_{\mathbb{T}_2},
\end{equation}
where $\mathbb{T}_1, \mathbb{T}_2$  be arbitrary time scales, $r(\boldsymbol{x}), s(\boldsymbol{x})$ be weights on $D$, and $y^{\boldsymbol\sigma}(\boldsymbol{x}) :=y(\sigma_1(x_1),\sigma_2(x_2))$ for all $\boldsymbol{x} = (x_1, x_2)\in D.$

We define 
\[ D_{\boldsymbol{x}} = \{\boldsymbol{t} \in D: \boldsymbol{a} \leq \boldsymbol{t} \leq \boldsymbol{x}\}, \quad \boldsymbol{a }=(a_1, a_2),   \]
\[ \bar{D}_{\boldsymbol{x}} = \{\boldsymbol{t} \in D: \boldsymbol{x} \leq \boldsymbol{t} \leq \boldsymbol{b}\}, \quad \boldsymbol{b }=(b_1, b_2),    \]
\[S(\boldsymbol{x}) = \int_{\bar{D}_{\boldsymbol{x}}}s(\boldsymbol{t})  \Delta \boldsymbol{t},\]
 and 
 \[S^*(x_1) = (b_2-a_2)[\sup_{\boldsymbol{x} \in D}S(\boldsymbol{x})](b_1-x_1) \quad \text{for }\quad \boldsymbol{x} = (x_1, x_2)\in D. \] 
For $\textbf{c} =(c_1, c_2)\in D$, we set
\[   
K_1(\boldsymbol{c}) =\bigg[ \int_{D_{\boldsymbol{c}}} \bigg(  \int_{D_{\boldsymbol{x}}}r^{-1}(\boldsymbol{x}) \Delta \boldsymbol{t} \bigg)(S^*(x_1)+S(\boldsymbol{x}))^2r^{-1}(\boldsymbol{x})  \Delta \boldsymbol{x}\bigg]^{\frac{1}{2}},
\]

\[   
L_1(\boldsymbol{c}) =\bigg[ \int_{\bar{D}_{\boldsymbol{c}}} \bigg(  \int_{\bar{D}_{\boldsymbol{x}}}r^{-1}(\boldsymbol{x}) \Delta \boldsymbol{t} \bigg)(S^*(x_1)+S(\boldsymbol{x}))^2r^{-1}(\boldsymbol{x})  \Delta \boldsymbol{x}\bigg]^{\frac{1}{2}},
\]
and
\[ M = \frac{\sqrt{2}}{2}\max\{K_1(\boldsymbol{c}), L_1(\boldsymbol{c})  \},  \]
where  $\boldsymbol{c} \in D$ is such that $D =D_{\boldsymbol{c}} \cup \bar{D}_{\boldsymbol{c}} $ and $|K_1(\boldsymbol{c}) -L_1(\boldsymbol{c})| \to \min$.

Now, fix $x_1 \in  [a_1,b_1]_{\mathbb{T}_1},$ and put
\[   
N_1(x_1,c'_2) =\bigg[ \int^{c'_2}_{a_2} \bigg(  \int^{x_2}_{a_2}r^{-1}(\boldsymbol{x}) \Delta t_2 \bigg)(S^*(x_1)+S(\boldsymbol{x}))^2r^{-1}(\boldsymbol{x})  \Delta  x_2\bigg]^{\frac{1}{2}},
\]

\[   
N_2(x_1,c'_2) =\bigg[ \int^{b_2}_{c'_2} \bigg(  \int^{b_2}_{x_2}r^{-1}(\boldsymbol{x}) \Delta  t_2 \bigg)(S^*(x_1)+S(\boldsymbol{x}))^2r^{-1}(\boldsymbol{x})  \Delta x_2\bigg]^{\frac{1}{2}},
\]
and
\[ N(x_1) = \frac{\sqrt{2}}{2}\max\{N_1(x_1,c'_2), N_2(x_1,c'_2)  \},  \]
where  $c'_2 \in [a_2,b_2]_{\mathbb{T}}$ is such that $|N_1(x_1,c'_2) -N_2(x_1,c'_2)| \to \min$.
\begin{theorem}\label{ap}
Suppose that  $y$ is a nontrivial solution of \eqref{Lya} such that 
\begin{equation}\label{boundary}
y(\textbf{x})|_{x_j =a_j}=y(\textbf{x})|_{x_j =b_j} = \frac{\partial y(\textbf{x})}{\Delta_i x_i}\bigg|_{x_j = a_j} =  \frac{\partial y(\textbf{x})}{\Delta_i x_i}\bigg|_{x_j = b_j}=0, \quad  i, j=1, 2,
\end{equation}
 then
\begin{equation}\label{eq:}
2M+\sup_{x_1 \in [a_1,b_1]_{\mathbb{T}_1}} [\mu_1(x_1)N(x_1)] \geq 1.
\end{equation}
\end{theorem} 

\begin{proof}
Multiplying \eqref{Lya}  by $y^{\boldsymbol\sigma}(\boldsymbol{x})$ and integrating both sides of \eqref{Lya}  with respect to $\boldsymbol{x}$ over $D$  we get
\begin{equation}\label{Mul}
\int_{D}\frac{\partial^{\boldsymbol{1}}}{\Delta \boldsymbol{x}^{\boldsymbol{1}}}   \bigg(r(\boldsymbol{x})\dfrac{\partial^{\boldsymbol{1}} y(\boldsymbol{x})}{\Delta\boldsymbol{x}^{\boldsymbol{1}}} \bigg)  y^{\boldsymbol\sigma}(\boldsymbol{x})  \Delta \boldsymbol{x} = -\int_{D} s(\boldsymbol{x})(y^{\boldsymbol\sigma}(\boldsymbol{x}))^{2}\Delta \boldsymbol{x}. 
\end{equation}
By  integrating by parts each side of \eqref{Mul}, we have
\[
\begin{split}
&\int_{D}\frac{\partial^{\boldsymbol{1}}}{\Delta \boldsymbol{x}^{\boldsymbol{1}}}   \bigg(r(\boldsymbol{x})\dfrac{\partial^{\boldsymbol{1}} y(\boldsymbol{x})}{\Delta\boldsymbol{x}^{\boldsymbol{1}}} \bigg) y^{\boldsymbol\sigma}(\boldsymbol{x})  \Delta \boldsymbol{x}\\
& = - \int^{b_1}_{a_1} \int^{b_2}_{a_2}
\frac{\partial}{\Delta_1 x_1} \bigg(r(\boldsymbol{x})\dfrac{\partial^{\boldsymbol{1}} y(\boldsymbol{x})}{\Delta\boldsymbol{x}^{\boldsymbol{1}}} \bigg) \frac{\partial y(\sigma_1(x_1),x_2)}{\Delta_2x_2}  \Delta x_2 \Delta x_1\\
&=\int_{D}r(\boldsymbol{x})\bigg(\frac{\partial^{\boldsymbol{1}} y(\boldsymbol{x})}{\Delta\boldsymbol{x}^{\boldsymbol{1}}} \bigg)^2 \Delta \boldsymbol{x}, \\
\end{split}
  \]
where, we have used  boundary conditions \eqref{boundary}, and
\[
\int_{D} s(\boldsymbol{x})(y^{\boldsymbol\sigma}(\boldsymbol{x}))^{2}\Delta \boldsymbol{x} = - \int_{D}\frac{\partial^{\boldsymbol{1}}S(\boldsymbol{x})}{\Delta \boldsymbol{x}^{\boldsymbol{1}}}(y^{\boldsymbol\sigma}(\boldsymbol{x}))^{2}\Delta \boldsymbol{x}= -\int_{D} S(\boldsymbol{x}) \frac{\partial^{\boldsymbol{1}} y^2(\boldsymbol{x})}{\Delta\boldsymbol{x}^{\boldsymbol{1}}} \Delta \boldsymbol{x},
\]
where, we have used $$y^2(\sigma_1(x_1),a_2) = y^2(\sigma_1(x_1),b_2) =y(\sigma_1(x_1),b_2)= \frac{\partial y^2(a_1,x_2)}{\Delta_2x_2}=\frac{\partial y^2(b_1,x_2)}{\Delta_2x_2}= 0.$$
Therefore, we obtain
\begin{equation}\label{A0}
\int_{D}r(\boldsymbol{x})\bigg(\frac{\partial^{\boldsymbol{1}} y(\boldsymbol{x})}{\Delta\boldsymbol{x}^{\boldsymbol{1}}} \bigg)^2 \Delta \boldsymbol{x}=\bigg|\int_{D} S(\boldsymbol{x}) \frac{\partial^{\boldsymbol{1}} y^2(\boldsymbol{x})}{\Delta\boldsymbol{x}^{\boldsymbol{1}}} \Delta \boldsymbol{x}\bigg|.
\end{equation}
Since
\[  
\frac{\partial^{\boldsymbol{1}} y^2(\boldsymbol{x})}{\Delta\boldsymbol{x}^{\boldsymbol{1}}} \Delta \boldsymbol{x} = \bigg(y^{\boldsymbol\sigma}(\boldsymbol{x})+ y(x_1, \sigma_2(x_2) \bigg) \frac{\partial y(\boldsymbol{x})}{\Delta_1 x_1}  +\bigg( \mu_1(x_1)\frac{\partial y(\boldsymbol{x})}{\Delta_1 x_1}  + 2y(\boldsymbol{x}) \bigg) \frac{\partial^{\boldsymbol{1}} y(\boldsymbol{x})}{\Delta\boldsymbol{x}^{\boldsymbol{1}}}   
\]
it follows that
\begin{equation}\label{A1}
\begin{split}
 \bigg|\int_{D} S(\boldsymbol{x}) \frac{\partial^{\boldsymbol{1}} y^2(\boldsymbol{x})}{\Delta\boldsymbol{x}^{\boldsymbol{1}}} \Delta \boldsymbol{x}\bigg|& \leq 
 \bigg|\int_{D} S(\boldsymbol{x})y^{\boldsymbol\sigma}(\boldsymbol{x})\frac{\partial y(\boldsymbol{x})}{\Delta_1 x_1}\Delta \boldsymbol{x}\bigg|\\
&\quad +\bigg|\int_{D} S(\boldsymbol{x})y(x_1, \sigma_2(x_2))\frac{\partial y(\boldsymbol{x})}{\Delta_1 x_1}\Delta \boldsymbol{x}\bigg|  \\
&\quad + \int_{D} S(\boldsymbol{x})\mu_1(x_1)\bigg|\frac{\partial y(\boldsymbol{x})}{\Delta_1 x_1}\bigg|\bigg|\frac{\partial^{\boldsymbol{1}} y(\boldsymbol{x})}{\Delta\boldsymbol{x}^{\boldsymbol{1}}}\bigg|\Delta \boldsymbol{x}  \\
&\quad + 2\int_{D} S(\boldsymbol{x})|y(\boldsymbol{x})|\bigg|\frac{\partial^{\boldsymbol{1}} y(\boldsymbol{x})}{\Delta\boldsymbol{x}^{\boldsymbol{1}}}\bigg|\Delta \boldsymbol{x}.  \\
\end{split}
\end{equation}
We have 
\begin{align}
\notag& \bigg|\int_{D} S(\boldsymbol{x})y(x_1, \sigma_2(x_2))\frac{\partial y(\boldsymbol{x})}{\Delta_1 x_1}\Delta \boldsymbol{x}\bigg|\\
 \notag&\leq  \int^{b_1}_{a_1}S^*(x_1) \bigg| \int^{b_2}_{a_2} y(x_1, \sigma_2(x_2))\frac{\partial y(\boldsymbol{x})}{\Delta_1 x_1}\Delta x_2\bigg| \Delta x_1  \\ 
\notag& =\int^{b_1}_{a_1} S^*(x_1) \bigg| \int^{b_2}_{a_2} y(x_1, x_2)\frac{\partial^{\boldsymbol{1}} y(\boldsymbol{x})}{\Delta\boldsymbol{x}^{\boldsymbol{1}}}  \Delta x_2\bigg| \Delta x_1    \\
&  =\int_{D}S^*(x_1)| y(\boldsymbol{x})|\bigg|\frac{\partial^{\boldsymbol{1}} (\boldsymbol{x})}{\Delta\boldsymbol{x}^{\boldsymbol{1}}}\bigg|\Delta \boldsymbol{x}, 
\label{A2}
\end{align}
where, we have used  
\[
\begin{split}
\int^{b_2}_{a_2} y(x_1, \sigma_2(x_2))\frac{\partial y(\boldsymbol{x})}{\Delta_1 x_1}\Delta x_2& =y(x_1, x_2)\frac{\partial y(\boldsymbol{x})}{\Delta_1 x_1}\bigg|^{x_2 = b_2}_{x_2=a_2} -  \int^{b_2}_{a_2} y(x_1, x_2)\frac{\partial^{\boldsymbol{1}} y(\boldsymbol{x})}{\Delta\boldsymbol{x}^{\boldsymbol{1}}}  \Delta x_2\\
&= -  \int^{b_2}_{a_2} y(x_1, x_2)\frac{\partial^{\boldsymbol{1}} y(\boldsymbol{x})}{\Delta \boldsymbol{x}^{\boldsymbol{1}}}  \Delta x_2.  \\  
\end{split}
\]

Similarly, we have
\begin{align}
\notag& \bigg|\int_{D} S(\boldsymbol{x})y^{\boldsymbol\sigma}(\boldsymbol{x})\frac{\partial y(\boldsymbol{x})}{\Delta_1 x_1}\Delta \boldsymbol{x}\bigg| \\ 
\notag&\leq \int_{D}S^*(x_1) |y(\sigma_1(x_1), x_2)|\bigg| \frac{\partial^{\boldsymbol{1}}y(\boldsymbol{x})}{\Delta\boldsymbol{x}^{\boldsymbol{1}}}\bigg|\Delta \boldsymbol{x}  \\
\notag& \leq \int_{D}S^*(x_1) \bigg| y(\boldsymbol{x}+ \mu_1(x_1) \frac{\partial y(\boldsymbol{x})}{\Delta_1 x_1}\bigg|\bigg|\frac{\partial^{\boldsymbol{1}}y(\boldsymbol{x})}{\Delta\boldsymbol{x}^{\boldsymbol{1}}}\bigg| \Delta \boldsymbol{x}   \\
\label{A3}& \leq\int_{D} S^*(x_1) |y(\boldsymbol{x})|\bigg|\frac{\partial^{\boldsymbol{1}}y(\boldsymbol{x})}{\Delta\boldsymbol{x}^{\boldsymbol{1}}}\bigg| \Delta \boldsymbol{x}\\
\notag&+ \int_{D}S^*(x_1)\mu_1(x_1)\bigg| \frac{\partial y(\boldsymbol{x})}{\Delta_1 x_1}\bigg|\bigg| \frac{\partial^{\boldsymbol{1}}y(\boldsymbol{x})}{\Delta\boldsymbol{x}^{\boldsymbol{1}}}\bigg| \Delta \boldsymbol{x}. 
\end{align}
From \eqref{A1}, \eqref{A2}, and \eqref{A3}, we get
\[   
\begin{split}
&\bigg| \int_{D} S(\boldsymbol{x}) \frac{\partial^{\boldsymbol{1}} y^2(\boldsymbol{x})}{\Delta\boldsymbol{x}^{\boldsymbol{1}}} \Delta \boldsymbol{x}\bigg|\\
& \leq 2 \int_{D}( S^*(x_1)+S(\boldsymbol{x}))|y(\boldsymbol{x})|\bigg|\frac{\partial^{\boldsymbol{1}} y(\boldsymbol{x})}{\Delta\boldsymbol{x}^{\boldsymbol{1}}}\bigg| \Delta \boldsymbol{x}\\
&\quad +  \int_{D}(S^*(x_1)+S(\boldsymbol{x}))\mu_1(x_1) \bigg|\frac{\partial y(\boldsymbol{x})}{\Delta_1 x_1}\bigg|\bigg| \frac{\partial^{\boldsymbol{1}}y(\boldsymbol{x})}{\Delta\boldsymbol{x}^{\boldsymbol{1}}}\bigg| \Delta \boldsymbol{x}.   \\
\end{split}
 \]
Applying inequality \eqref{ptMIN} with $n=2, \alpha =\beta =1, \boldsymbol{\lambda} =\boldsymbol{1}$, $\omega(\boldsymbol{x}) = S^*(x_1)+ S(\boldsymbol{x})$, and $\tau (\boldsymbol{x}) = r(\boldsymbol{x})$, we get
\begin{equation}\label{B1}
 \int_{D}(  S^*(x_1)+S(\boldsymbol{x}))|y(\boldsymbol{x})|\bigg|\frac{\partial^{\boldsymbol{1}} y(\boldsymbol{x})}{\Delta\boldsymbol{x}^{\boldsymbol{1}}}\bigg| \Delta \boldsymbol{x}\leq M \int_{D}r(\boldsymbol{x})\bigg(\frac{\partial^{\boldsymbol{1}} y(\boldsymbol{x})}{\Delta\boldsymbol{x}^{\boldsymbol{1}}}\bigg)^2 \Delta \boldsymbol{x}. 
 \end{equation}
Applying inequality \eqref{ptMIN} with $n=1, \alpha =\beta =1, \boldsymbol{\lambda} =1$, $\omega(\boldsymbol{x}) = S^*(x_1)+ S(\boldsymbol{x})$, and $\tau (\boldsymbol{x}) = r(\boldsymbol{x})$,   we have
\[
    \int^{b_2}_{a_2}(S^*(x_1)+S(\boldsymbol{x})) \bigg|\frac{\partial y(\boldsymbol{x})}{\Delta_1 x_1}\bigg|\bigg| \frac{\partial^{\boldsymbol{1}}y(\boldsymbol{x})}{\Delta\boldsymbol{x}^{\boldsymbol{1}}}\bigg| \Delta x_2  \leq N(x_1)\int^{b_2}_{a_2}r(\boldsymbol{x}) \bigg( \frac{\partial^{\boldsymbol{1}}y(\boldsymbol{x})}{\Delta\boldsymbol{x}^{\boldsymbol{1}}}\bigg)^2\Delta x_2.
\]
This implies that
\begin{align}
 \notag&\int^{b_1}_{a_1} \int^{b_2}_{a_2}(S^*(x_1)+S(\boldsymbol{x}))\mu_1(x_1) \bigg|\frac{\partial y(\boldsymbol{x})}{\Delta_1 x_1}\bigg|\bigg| \frac{\partial^{\boldsymbol{1}}y(\boldsymbol{x})}{\Delta\boldsymbol{x}^{\boldsymbol{1}}}\bigg| \Delta x_2 \Delta x_1 \\
\notag& \leq \int^{b_1}_{a_1} \mu_1(x_1)N(x_1)\int^{b_2}_{a_2}r(\boldsymbol{x}) \bigg( \frac{\partial^{\boldsymbol{1}}y(\boldsymbol{x})}{\Delta\boldsymbol{x}^{\boldsymbol{1}}}\bigg)^2\Delta x_2 \Delta x_1\\
\label{B2}& \leq \bigg(\sup_{x_1 \in [a_1,b_1]_{\mathbb{T}_1}} [\mu_1(x_1)N(x_1)] \bigg)\int_{D}r(\boldsymbol{x})\bigg(\frac{\partial^{\boldsymbol{1}} y(\boldsymbol{x})}{\Delta\boldsymbol{x}^{\boldsymbol{1}}}\bigg)^2 \Delta \boldsymbol{x}.
\end{align}
Substituting \eqref{B1} and \eqref{B2} into \eqref{A0}, we have 
\[    
\int_{D}r(\boldsymbol{x})\bigg(\frac{\partial^{\boldsymbol{1}} y(\boldsymbol{x})}{\Delta\boldsymbol{x}^{\boldsymbol{1}}} \bigg)^2 \Delta \boldsymbol{x} \leq \bigg(2M+\sup_{x_1 \in [a_1,b_1]_{\mathbb{T}_1}} [\mu_1(x_1)N(x_1)] \bigg)
\int_{D}r(\boldsymbol{x})\bigg(\frac{\partial^{\boldsymbol{1}} y(\boldsymbol{x})}{\Delta\boldsymbol{x}^{\boldsymbol{1}}} \bigg)^2 \Delta \boldsymbol{x}.
\]
Then, we get
\[2M + \sup_{x_1 \in [a_1,b_1]_{\mathbb{T}_1}} [\mu_1(x_1)N(x_1)] \geq 1.  \]
\qed
\end{proof}

In Theorem \ref{ap}  if $r(\boldsymbol{x})  \equiv 1$, then we have the following result.

\begin{corollary}\label{Co1}
Suppose that  $y$ is a nontrivial solution of
\[
\frac{\partial^{\boldsymbol{1}}}{\Delta \textbf{x}^{\boldsymbol{1}}}   \bigg(\dfrac{\partial^{\boldsymbol{1}} y(\textbf{x})}{\Delta\textbf{x}^{\boldsymbol{1}}} \bigg)  +s(\textbf{x}) y^{\boldsymbol\sigma}(\textbf{x}) =0, \quad \boldsymbol{x} \in D
\]
which satisfies boundary conditions \eqref{boundary}, then
\[
2 \max\{K_2(\textbf{c}), L_2(\textbf{c})  \} +\sup_{x_1 \in [a_1,b_1]_{\mathbb{T}_1}}\bigg( \mu_1(x_1) \max\{N_3(x_1,c'_2), N_4(x_1,c'_2)  \} \bigg) \geq \sqrt{2},
\]
where
\[   
K_2(\textbf{c}) =\bigg[ \int_{D_{\textbf{c}}} (x_1-a_1)(x_2-a_2)[S^*(x_1)+S(\textbf{x})]^2 \Delta \textbf{x}\bigg]^{\frac{1}{2}},
\]
\[   
L_2(\textbf{c}) =\bigg[ \int_{\bar{D}_{\textbf{c}}} (b_1-x_1)(b_2-x_2)[S^*(x_1)+ S(\textbf{x})]^2    \Delta \textbf{x}\bigg]^{\frac{1}{2}},
\]
with $\textbf{c} \in D$ is such that $D = D_{\textbf{c}} \cup \bar{D}_{\textbf{c}}$ and $|K_2(\textbf{c}) -L_2(\textbf{c})| \to \min,$ and 
\[   
N_3(x_1,c'_2) =\bigg[ \int^{c'_2}_{a_2} (x_2-a_2) [S^*(x_1)+ S(\textbf{x})]^2  \Delta  x_2\bigg]^{\frac{1}{2}},
\]
\[   
N_4(x_1,c'_2) =\bigg[ \int^{b_2}_{c'_2} (b_2-x_2)  [S^*(x_1)+ S(\textbf{x})]^2  \Delta x_2\bigg]^{\frac{1}{2}},
\]
for   $x_1 \in  [a_1,b_1]_{\mathbb{T}_1}$,  $c'_2 \in [a_2,b_2]_{\mathbb{T}}$ is such that $|N_3(x_1,c'_2) -N_4(x_1,c'_2)| \to \min$.
\end{corollary}

As a special case of Corollary \ref{Co1}, when $\mathbb{T}_1 = \mathbb{T}_2 = \mathbb{R}$,  we see that if $\boldsymbol{c}= (\frac{a_1+b_1}{2}, b_2) $ and  $\boldsymbol{c}= (a_1, \frac{a_2+b_2}{2})$ are in $  \in [a_1,b_1]\times[a_2,b_2]$, then $|K_2(\boldsymbol{c}) -L_2(\boldsymbol{c})| = 0$, and $\mu_1(x_1) = 0$, we have the following result.
\begin{corollary}
Suppose that  $y$ is a nontrivial solution of
\[
\frac{\partial^{\boldsymbol{1}}}{\partial \textbf{x}^{\boldsymbol{1}}}   \bigg(\dfrac{\partial^{\boldsymbol{1}} y(\textbf{x})}{\partial\textbf{x}^{\boldsymbol{1}}} \bigg) + s(\textbf{x})y(\textbf{x}) =0, \quad  \boldsymbol{x} \in [a_1,b_1]\times[a_2,b_2]
\]
which satisfies \eqref{boundary}, then
\[
   \int_{D} (x_1-a_1)(x_2-a_2)[S^*(x_1)+S(\textbf{x})]^2  d\textbf{x}  \geq 1.
\]
\end{corollary}

\begin{remark}
We see that the sufficient condition that equation \eqref{Lya} does not have a nontrivial solution $y$ such that
\[
y(\boldsymbol{x})|_{x_j =a_j}=y(\boldsymbol{x})|_{x_j =b_j} = \frac{\partial y(\boldsymbol{x})}{\Delta_i x_i}\bigg|_{x_j = a_j} =  \frac{\partial y(\boldsymbol{x})}{\Delta_i x_i}\bigg|_{x_j= b_j}=0, \quad i, j=1, 2,
\]
can be obtained from Theorem \ref{ap}.
\end{remark}

Now, we consider the following half-linear delay dynamic equation
\begin{equation}\label{delay}
\frac{\partial^{\boldsymbol{1}}} {\Delta \boldsymbol{x}^{\boldsymbol{1}}}  \bigg(r(\boldsymbol{x})\dfrac{\partial^{\boldsymbol{1}} y(\boldsymbol{x})}{\Delta\boldsymbol{x}^{\boldsymbol{1}}} \bigg) + s(\boldsymbol{x})y(\theta(\boldsymbol{x})) =0 \quad \text{on } \quad  D,
\end{equation}
where $\theta(\boldsymbol{x}) = (\theta_1(x_1), \theta_2(x_2))$ for all $\boldsymbol{x} = (x_1,x_2) \in D, $ $\theta_j : \mathbb{T}_j \to \mathbb{T}_j$, $\theta_j(x_j) \leq x_j$, and $\lim_{x_j \to \infty}\theta_j(x_j) = \infty$ for  $j =1, 2.$
\begin{corollary}
Suppose that  $y$ is a nontrivial solution of equation \eqref{delay} which satisfies boundary conditions \eqref{boundary} and $\frac{\partial y(\boldsymbol{x})}{\Delta_1 x_1}, $ $ \frac{\partial y(\boldsymbol{x})}{\Delta_2 x_2},$   do not change sign in $(a_1,b_1)_{\mathbb{T}_1}\times(a_2,b_2)_{\mathbb{T}_2}$, then 
\[
2M+\sup_{x_1 \in [a_1,b_1]_{\mathbb{T}_1}}[ \mu_1(x_1)N(x_1)] \geq 1.
\]

\end{corollary}

\begin{proof}
Since $\frac{\partial y(\boldsymbol{x})}{\Delta_1 x_1},  \frac{\partial y(\boldsymbol{x})}{\Delta_2 x_2}$ do not change sign in $(a_1,b_1)_{\mathbb{T}_1}\times(a_2,b_2)_{\mathbb{T}_2}$ we can assume that \eqref{delay} has a solution $y$ satisfying $\frac{\partial y(\boldsymbol{x})}{\Delta_1 x_1}>0,  \frac{\partial y(\boldsymbol{x})}{\Delta_2 x_2}>0$. Therefore, since $\theta(\boldsymbol{x}) \leq \boldsymbol{x}$, that $y(\theta(\boldsymbol{x})) \leq y^{\boldsymbol\sigma}(\boldsymbol{x})$.
The remainder of the proof is similar to the proof of Theorem \ref{ap}  and hence  omitted.\qed
\end{proof}

Finally, we give an upper bound of the solutions of the following integro-partial dynamic equation
\begin{equation}\label{par}
\dfrac{\partial^{\boldsymbol{1}} y(\boldsymbol{x})}{\Delta  \boldsymbol{x}^{\boldsymbol{1}}} = \zeta\bigg(\boldsymbol{x}, \frac{\partial^{\boldsymbol{1}} y(\boldsymbol{x})}{\Delta  \boldsymbol{x}^{\boldsymbol{1}}}, I(y(\boldsymbol{x}))\bigg), \quad \boldsymbol{x} \in  \Omega,
\end{equation}
with the initial conditions $y(\boldsymbol{x})|_{x_j =a_j} = 0$ for  all  $j \in [1, n]_{\mathbb{N}}$, where 
\[I(y(\boldsymbol{x})) = \int_{\Omega_{\boldsymbol{x}}}\Phi \bigg(\boldsymbol{t}, y(\boldsymbol{t}),\frac{\partial^{\boldsymbol{1}} y(\boldsymbol{t})}{\Delta  \boldsymbol{t}^{\boldsymbol{1}}} \bigg) \Delta \boldsymbol{t}.  \]

\begin{theorem}
Suppose that  $\Phi (\textbf{x}, y(\textbf{x}),\frac{\partial^{\boldsymbol{1}} y(\textbf{x})}{\Delta  \textbf{x}^{\boldsymbol{1}}}) \leq \omega(\textbf{x}) |y(\textbf{x})|^\beta | \frac{\partial^{\boldsymbol{1}} y(\textbf{x})}{\Delta  \textbf{x}^{\boldsymbol{1}}}|^\alpha$,
where $\omega \in \mathcal{W}(\Omega)$, and
\begin{equation}\label{ptzeta}
\bigg|\zeta\bigg(\textbf{x}, \frac{\partial^{\boldsymbol{1}   } y(\textbf{x})}{\Delta  \textbf{x}^{\boldsymbol{1}}}, I(y(\textbf{x}))\bigg)\bigg| \leq w_1(\textbf{x})+ w_2(\textbf{x})\bigg| \frac{\partial^{ \boldsymbol{1} } y(\textbf{x})}{\Delta  \textbf{x}^{\boldsymbol{1}}} \bigg|^\gamma + w_3(\textbf{x})|I(y(\textbf{x}))|, 
\end{equation}
where $\gamma \in (0; 1)$, $w_1, w_2,$ and $w_3$ are non-negative on $\Omega$, $w_2(\textbf{x}) < 1$ for all $\textbf{x}\in \Omega$. If  equation \eqref{par} has a 
solution  $y \in \mathcal{L}^{\alpha+\beta}_{\textbf{a}}(\Omega, 1, \boldsymbol{1})$, then
\begin{equation}\label{dgy}
y(\textbf{x}) \leq \int_{\Omega_{\textbf{x}}}[A^{1-\alpha-\beta}(\textbf{t}) +(1-\alpha-\beta)B(\textbf{t}) \Vol(\Omega_{\textbf{t}}) ]^{\frac{1}{1-\alpha-\beta}}   \Delta \textbf{t}
\end{equation}
as long as the right-hand side integral exists, where  
\[B(\textbf{x})  = \sup_{\textbf{t }\in\Omega_{\textbf{x}}}{\dfrac{w_3(\textbf{t})K(\textbf{t})}{1-w_2(\textbf{t})} },\quad A(\textbf{x})  = \sup_{\textbf{t }\in \Omega_{\textbf{x}}}{\dfrac{w_1(\textbf{t}) + (1-\gamma)\gamma^{\frac{\gamma}{1-\gamma}}}{1-w_2(\textbf{t})} },\] 
 $\Vol(\Omega_{\textbf{t}})$ is the volume of a rectangular region $\Omega_{\textbf{t}},$ and
 \[
K(\textbf{x}) =\bigg(\frac{\alpha}{\alpha+\beta}\bigg)^{\frac{\alpha}{\alpha+\beta}}  \bigg(\int_{\Omega_{\textbf{x}}} (\Vol(\Omega_{\textbf{t}}))^{\alpha+\beta -1}\omega^{\frac{\alpha+\beta}{\beta}}(\textbf{t})  \Delta \textbf{t}\bigg)^{\frac{\beta}{\alpha+\beta}}\quad \text{for} \quad \textbf{x} \in \Omega.
 \]
\end{theorem}

\begin{proof}
By applying inequality \eqref{ptmin1} with $\boldsymbol{\lambda} = \boldsymbol{1}, \tau \equiv 1$, we have
\begin{equation}\label{ptO1}
\int_{\Omega_{\boldsymbol{x}}} \omega(\boldsymbol{t})|y(\boldsymbol{t})|^\beta\bigg| \dfrac{\partial^{\boldsymbol{1}} y(\boldsymbol{t})}{\Delta  \boldsymbol{t}^{\boldsymbol{1}}} \bigg|^\alpha   \Delta \boldsymbol{t} \leq K(\boldsymbol{x}) \int_{\Omega_{\boldsymbol{x}}}\bigg| \dfrac{\partial^{\boldsymbol{1}} y(\boldsymbol{t})}{\Delta  \boldsymbol{t}^{\boldsymbol{1}}} \bigg|^{\alpha+\beta}   \Delta \boldsymbol{t}
\end{equation}
for $ \boldsymbol{x} \in \Omega.$
We consider the function $f(x) = x^\gamma- x -(1-\gamma)\gamma^{\frac{\gamma}{1-\gamma}}$ for $\gamma \in (0,1)$ and $x\geq 0$. We see that $f(x)$ obtains its maximum at $x =\gamma^{\frac{1}{1-\gamma}}$ and $f_{\text{max}} = 0.$ Then, we have
\begin{equation}\label{ptO2}
 \bigg| \dfrac{\partial^{\boldsymbol{1}} y(\boldsymbol{x})}{\Delta  \boldsymbol{x}^{\boldsymbol{1}}} \bigg|^{\gamma} \leq \bigg|\dfrac{\partial^{\boldsymbol{1}} y(\boldsymbol{x})}{\Delta  \boldsymbol{x}^{\boldsymbol{1}}}\bigg| + (1-\gamma)\gamma^{\frac{\gamma}{1-\gamma}}, \quad \boldsymbol{x} \in \Omega. 
\end{equation}
Substituting \eqref{ptO1} and \eqref{ptO2} into \eqref{ptzeta}, we obtain
\[(1-w_2(\boldsymbol{x}))  \bigg| \dfrac{\partial^{\boldsymbol{1}} y(\boldsymbol{x})}{\Delta  \boldsymbol{x}^{\boldsymbol{1}}} \bigg| \leq  w_1(\boldsymbol{x}) + (1-\gamma)\gamma^{\frac{\gamma}{1-\gamma}} +  w_3(\boldsymbol{x})K(\boldsymbol{x}) \int_{\Omega_{\boldsymbol{x}}}\bigg| \dfrac{\partial^{\boldsymbol{1}} y(\boldsymbol{t})}{\Delta  \boldsymbol{t}^{\boldsymbol{1}}} \bigg|^{\alpha+\beta}   \Delta \boldsymbol{t}     \]
for $ \boldsymbol{x} \in \Omega.$
This implies that
\begin{equation}\label{ptss}
\bigg| \dfrac{\partial^{\boldsymbol{1}} y(\boldsymbol{x})}{\Delta  \boldsymbol{x}^{\boldsymbol{1}}} \bigg| \leq \frac{w_1(\boldsymbol{x}) + (1-\gamma)\gamma^{\frac{\gamma}{1-\gamma}}}{1-w_2(\boldsymbol{x})}+ \frac{w_3(\boldsymbol{x})K(\boldsymbol{x})}{1-w_2(\boldsymbol{x})}\int_{\Omega_{\boldsymbol{x}}}\bigg| \dfrac{\partial^{\boldsymbol{1}} y(\boldsymbol{t})}{\Delta  \boldsymbol{t}^{\boldsymbol{1}}} \bigg|^{\alpha+\beta}   \Delta \boldsymbol{t}
\end{equation} 
for $ \boldsymbol{x} \in \Omega.$
Let $\boldsymbol{s} \in \Omega$  be arbitrary, but fixed; then inequality \eqref{ptss} gives
\begin{equation}\label{ptRR}
\bigg| \dfrac{\partial^{\boldsymbol{1} }y(\boldsymbol{t})}{\Delta  \boldsymbol{t}^{\boldsymbol{1}}} \bigg| \leq A(\boldsymbol{s}) + B(\boldsymbol{s})\int_{\Omega_{\boldsymbol{t}}}\bigg| \dfrac{\partial^{\boldsymbol{1}} y(\boldsymbol{u})}{\Delta  \boldsymbol{u}}^{\boldsymbol{1}} \bigg|^{\alpha+\beta}   \Delta \boldsymbol{u}, \quad \boldsymbol{t} \in \Omega_{\boldsymbol{s}}.   
\end{equation}
Next, let $R(\boldsymbol{t})$ be the right-hand side of  \eqref{ptRR}; then
\[ \dfrac{\partial^{\boldsymbol{1}} R(\boldsymbol{t})}{\Delta  \boldsymbol{t}^{\boldsymbol{1}}} = B(\boldsymbol{s}) \bigg| \dfrac{\partial^{\boldsymbol{1}} y(\boldsymbol{t})}{\Delta  \boldsymbol{t}^{\boldsymbol{1}}} \bigg|^{\alpha+\beta} \leq B(\boldsymbol{s}) R^{\alpha+\beta}(\boldsymbol{t}), \quad \boldsymbol{t} \in \Omega_{\boldsymbol{s}},   \]
where $ R|_{t_j = a_j} = A(\boldsymbol{s})$ for all $j \in [1, n]_{\mathbb{N}}$. Since
\[ \int_{\Omega_{\boldsymbol{z}}} \frac{1}{ {R^{\alpha+\beta}(\boldsymbol{t})} }\frac{\partial^{\boldsymbol{1}} R(\boldsymbol{t})}{\Delta  \boldsymbol{t}^{\boldsymbol{1}}} \Delta \boldsymbol{t} \geq \frac{1}{1-\alpha-\beta}[ R^{1-\alpha -\beta}(\boldsymbol{z})- A^{1-\alpha-\beta}(\boldsymbol{s}) ], \quad \boldsymbol{z} \in \Omega_{\boldsymbol{s}},    \]
then
\[ \frac{1}{1-\alpha-\beta}[ R^{1-\alpha -\beta}(\boldsymbol{z})- A^{1-\alpha-\beta}(\boldsymbol{s})] \leq B(\boldsymbol{s}) \Vol(\Omega_{\boldsymbol{z}}), \quad \boldsymbol{z} \in \Omega_{\boldsymbol{s}}.    \]
Therefore,
\[ \bigg|\dfrac{\partial^{\boldsymbol{1}} y(\boldsymbol{z})}{\Delta  \boldsymbol{z}^{\boldsymbol{1}}}\bigg| \leq  R(\boldsymbol{z}) \leq   [A^{1-\alpha-\beta}(\boldsymbol{s}) +(1-\alpha-\beta)B(\boldsymbol{s}) \Vol(\Omega_{\boldsymbol{z}}) ]^{\frac{1}{1-\alpha-\beta}}, \quad \boldsymbol{z} \in \Omega_{\boldsymbol{s}}. \]
In the above inequality replacing $\boldsymbol{z}$ by $\boldsymbol{s}$ and integrating both sides with respect
to $\boldsymbol{s}$ over $\Omega_{\boldsymbol{x}}$ for $\boldsymbol{x} \in \Omega$  we obtain \eqref{dgy}.\qed
\end{proof}

\begin{acknowledgements}
The author would like to express his deepest gratitude to  Assoc. Prof. Dinh Thanh Duc,  Prof. Vu Kim Tuan and Nguyen Du Vi Nhan  for their comments and suggestions  to improve this paper. 
\end{acknowledgements}



\end{document}